\documentclass[microtype]{gtpart}     

\author{Ilya Grigoriev}
\givenname{Ilya}
\surname{Grigoriev}
\address{Department of Mathematics\\
University of Chicago\\\newline
5734 S University Ave\\
Chicago, IL 60637}
\email{ilyag@uchicago.edu}
\urladdr{http://math.uchicago.edu/~ilyagr/}

\keyword{Manifold bundles}
\keyword{Characteristic classes}
\keyword{tautological ring}
\keyword{Miller-Morita-Mumford classes}

\subject{primary}{msc2010}{55R40}
\subject{secondary}{msc2010}{57R22}
\subject{secondary}{msc2010}{55T10} 

\arxivreference{1310.6804}
\arxivpassword{jztgx}

\newcommand{\notingt}[1]{}
\newcommand{\mybibstyle}{\bibliographystyle{gtart}}

\usepackage{verbatim}  
 \newcommand{\secone}{\section}
\newcommand{\sectwo}{\subsection}

\usepackage{amsmath}
\usepackage{amsthm}
\usepackage{amssymb}
 \usepackage[all, pdf]{xy}
 \usepackage{enumerate}
 \newcommand{\faktor}[2]{#1/#2}

 \usepackage{graphicx}
\usepackage{color}
\usepackage{transparent}
\usepackage{url}

\usepackage{hyphenat}

 \notingt{
    \usepackage[breaklinks]{hyperref} 
    \usepackage{xcolor}
    \definecolor{dark-red}{rgb}{0.4,0.15,0.15}
    \definecolor{dark-blue}{rgb}{0.15,0.15,0.4}
    \definecolor{medium-blue}{rgb}{0,0,0.5}
    \hypersetup{
        colorlinks, linkcolor={dark-red},
            citecolor={dark-blue}, urlcolor={medium-blue}
    }
}

\numberwithin{equation}{subsection}
  \newtheorem{theorem}{Theorem}[section]
\newtheorem{proposition}[theorem]{Proposition}
\newtheorem{lemma}[theorem]{Lemma}
\newtheorem{corollary}[theorem]{Corollary}
\newtheorem{observation}[theorem]{Observation}
\newtheorem{fact}[theorem]{Fact}
\newtheorem*{fact*}{Fact}
\theoremstyle{definition}
\newtheorem{remark}[theorem]{Remark}
\newtheorem{definition}[theorem]{Definition}
\newtheorem{example}[theorem]{Example}
\newtheorem{procedure}[theorem]{Procedure}
\newtheorem*{notation}{Notation}

 \theoremstyle{plain}  \newcommand{\thistheoremname}{}
\newtheorem{genericthm}[theorem]{\thistheoremname}

\newtheorem*{genericthm*}{\thistheoremname}
\newenvironment{namedtheorem*}[1]
  {\renewcommand{\thistheoremname}{#1}    \begin{genericthm*}}
  {\end{genericthm*}}
            \DeclareMathOperator{\image}{image}
\DeclareMathOperator{\Aut}{Aut}

\newcommand{\tH}{\mathcal{H}}
\newcommand{\tA}{\mathcal{A}} \newcommand{\tB}{\mathcal{B}}
\newcommand{\toiso}{\stackrel{\sim}{\longrightarrow}}
\newcommand{\coeff}[1]{#1_{\mathrm{coeff}}}
\newcommand{\bigsharp}{\mathop{\#}}
\newcommand{\sumspheres}[1]{\bigsharp_g S^{#1} \times S^{#1}}
\newcommand{\BDiff}{\operatorname{BDiff}}
\newcommand{\EDiff}{\operatorname{EDiff}}
\newcommand{\rank}{\operatorname{rank}}
\newcommand{\Hom}{\operatorname{Hom}} 
\newcommand{\BSO}[1]{BSO_{#1}}
\newcommand{\QQ}{\mathbb{Q}}
\newcommand{\ZZ}{\mathbb{Z}}

\newcommand{\NN}{\mathbb{N}}
\newcommand{\floor}[1]{\left\lfloor#1\right\rfloor}
\newcommand{\ceil}[1]{\left\lceil#1\right\rceil}

\newcommand{\gkappa}[1]{\kappa_{#1}}  \newcommand{\Mg}{\cM_g} \newcommand{\cM}{\mathcal{M}} \newcommand{\Mgn}[1]{\cM_{g}\!\left(#1\right)} 
\newcommand{\Mgtot}{\mathcal{E}_g}

\newcommand{\hquot}{/\!\!/}
\newcommand{\Diff}{\operatorname{Diff}}
 
\newcommand{\bundle}[2]{\substack{#1 \\ \downarrow \\ #2}} 
 \newcommand{\pbundle}[2]{\!\left(\bundle{#1}{#2}\right)}

  \newcommand{\defterm}[1]{\emph{#1}}

\newcommand{\tbasisl}{\mathcal{S}}
\newcommand{\tbasiss}{\mathcal{S}'}
\newcommand{\taut}{\mathcal{R}}
\newcommand{\mpoint}{\star}  
\newcommand{\Alt}[1]{\operatorname{Alt} #1}
\newcommand{\K}[1]{\left(\ker \pi_! \right)^{#1}}
\newcommand{\tC}{\mathcal{C}}\newcommand{\tD}{\mathcal{D}}

\newcommand{\gpsi}[2]{{#1}_{\left(#2\right)}} 
\newcommand{\intclass}[1]{\nu_{\left(#1\right)}}   \newcommand{\cA}{\mathcal{A}} \newcommand{\cB}{\mathcal{B}} \newcommand{\cD}{\mathcal{D}}

\begin{document}

\title[Characteristic Classes of Manifold Bundles]{Relations among Characteristic Classes of Manifold Bundles
}

\notingt{
    \subjclass[2010]{55R40,      57R22,  	    55T10      }

    \keywords{Manifold bundles,      tautological ring,      Miller-Morita-Mumford classes}
}

\begin{abstract}

We study relations among characteristic classes of smooth manifold bundles with highly-connected fibers. For bundles with fiber the connected sum of $g$ copies of a product of spheres $S^d \times S^d$ and odd $d$, we find numerous algebraic relations among so-called ``generalized Miller-Morita-Mumford classes". For all $g > 1$, we show that these infinitely many classes are algebraically generated by a finite subset.

Our results contrast with the fact that there are no algebraic relations among these classes in a range of cohomological degrees that grows linearly with $g$, according to recent homological stability results. In the case of surface bundles ($d=1$), our approach recovers some previously known results about the structure of the classical ``tautological ring", as introduced by Mumford,  using only the tools of algebraic topology.

\end{abstract}

\maketitle

   \notingt{
    \vspace{-1.3cm}
    \tableofcontents
    \vspace{-1.3cm}
}

\secone{Introduction}

Let $M$ be a $2d$-dimensional closed oriented smooth manifold. We denote by $\Diff M$ the topological group of \emph{orientation-preserving} diffeomorphisms of $M$. The bar construction can be used to construct the space $\BDiff(M)$ that classifies bundles with fiber $M$. For any characteristic class of vector bundles $p \in H^{*+2d}(\BSO{2d}; \QQ)$, we will define a \defterm{generalized Miller-Morita-Mumford class} (or just \defterm{kappa class}) $\gkappa{p} \in H^*\left(\BDiff(M); \QQ \right)$. These are the simplest examples of characteristic classes of bundles\footnote{A geometric example of such a bundle is a proper submersion $f: E\to B$ of smooth, oriented manifolds that has $M$ as its fiber.}
with fiber $M$ and structure group $\Diff M$.

We are mainly interested in the case where the fiber is $\sumspheres{d}$, the connected sum of $g$ copies of $S^d \times S^d$. More generally, we let the fiber to be a highly-connected manifold (see Definition~\ref{def:highly-connected-manifold}) of genus $g$ and dimension $2d$, denoted $M_g^{2d}$ or $M_g$.  Recall that $H^*(\BSO{2d}; \QQ) = \QQ[p_1, \ldots, p_{d-1}, e]$, where $p_i$ is the Pontryagin class of degree $4i$ and $e$ is the Euler class of degree $2d$. Let $\tbasisl \subset H^*(\BSO{2d}; \QQ)$ consist of the monomials in  the Pontryagin classes and the Euler class. For each such monomial, there is a corresponding MMM class in $H^*\left( \BDiff M_g; \QQ \right)$, which gives rise to a map 
\begin{equation*}
    \taut_{d}\co \QQ[\gkappa{p} \mid p \in \tbasisl] \to H^*(\BDiff M_g; \QQ).  \end{equation*}

This paper presents a large family of polynomials in the MMM classes that lies in the kernel of the map $\taut_d$, in the case that $d$ is odd. In the $d > 1$ case, ours are the first results of this kind. In the $d=1$ case, we recover previously known results, but using purely homotopy theoretic methods. Our first main result is the following.

\begin{theorem}\label{thm:finite-generation}
    The image of $\taut_d$ is finitely-generated as a $\QQ$-algebra when $d$ is odd and $g > 1$.
\end{theorem}

 In Proposition~\ref{pro:taut-mod-radical-generators}, we also show that for all odd $d$, the Krull dimension of the image of $\taut_d$ is at most $2d$.

Our methods generalize the technique Randal-Williams developed for the $d=1$ case in~\cite{oscar12-rel-revisited}, which in turn is based on the work of Morita~\cite{morita89-families1}. They allow us to present many specific elements in $\ker \taut_d$. For instance, Randal-Williams found various relations among the images of the  classes 
\begin{equation*}
    \kappa_i := \gkappa{e^{i+1}} \in H^{2di}\left( \BDiff M_g ;\QQ \right)
\end{equation*}
under the map $R_d$ in the case when $d = 1$. We find that the same relations hold for any odd $d$ (see section~\ref{sec:all-oscar-relations} for details and examples). This is surprising, as no map between subrings of $H^*(\BDiff M^{2d}_g)$ for different $d$ that takes $\kappa_i$ to $\kappa_i$ can preserve the grading on the cohomology.

  \sectwo{Manifolds with a fixed disk and homological stability}\label{sec:intro-fixed-disks}
Let $\tbasiss \subset \tbasisl$ be the set of monomials in the classes\footnote{
We use the notation $\ceil{\cdot}$ and $\floor{\cdot}$ for rounding up and down (respectively) to the nearest integer.
} $p_{\ceil{\frac{d+1}{4}}}$, $p_{\ceil{\frac{d+1}{4}}+1}$, ..., $p_{d-1}$, and $e$ of total degree greater than $2d$. Let $\taut'_d$ denote the map $\taut_d$ restricted to $\QQ[\gkappa{p} \mid p \in \tbasiss]$.  Our second main result is

\begin{theorem}\label{thm:lowest-relation-prime}
    If $d \equiv 3\pmod 4$, the map $\taut'_d$ has nontrivial kernel in degree $2g+2$. If $d \equiv 1 \pmod 4$, the map $\taut'_d$ has nontrivial kernel in degree $6g+6$.
\end{theorem}

By contrast, the map $\taut'_d$ is known to be injective in a range of cohomological degrees $* \leq (g-4)/2$ when the fiber is $\sumspheres{d}$ and $d \neq 2$.  This fact and the related phenomenon of \defterm{homological stability} are a large part of the motivation for our work. We now describe them in more detail.

  Let $\Diff(M_g, D^{2d}) \subset \Diff(M_g)$ the subgroup of those diffeomorphisms that fix pointwise a chosen disk in $M_g$, and let $f\co \BDiff(M_g,D^{2d}) \to \BDiff(M_g)$ be the map induced on the bar constructions by the inclusion of groups.  We define the map $\taut_{\delta, d}\co \QQ[\gkappa{p} \mid p \in \tbasiss] \to H^*(\BDiff(M_g, D^{2d}); \QQ)$ as the map that makes the following diagram commute. (The $\delta$ stands for ``fixed disk''. See Appendix~\ref{sec:low-pontryagins} for a comparison of the images of the various maps in the diagram.)
\begin{equation}\label{eq:taut-taut-prime-diagram}
\xymatrix{
  \QQ[ \gkappa{p} \mid p \in \tbasisl] \ar[rr]^{\taut_d} && H^*(\BDiff M_g;\QQ ) \ar[d]^{f^*} \\
  \QQ[ \gkappa{p} \mid p \in \tbasiss]  \ar[rr]^{\taut_{\delta,d}} \ar[urr]^{\taut'_d} \ar@{^(->}[u]_{i}&& H^*\left(\BDiff\left(M_g, D^{2d}\right);\QQ \right)
}
\end{equation}

The following fact, in the $d=1$ case, is a consequence of the Madsen-Weiss theorem~\cite{madsen-weiss07} and the Harer stability theorem~\cite{harer85-stability}, with the improved stability range by Boldsen~\cite{boldsen12}. In the case when $d>2$, the fact is a consequence of two theorems of Galatius and Randal-Williams~\cite{soren-oscar12-limit, soren-oscar12-stability}.

\begin{fact}\label{fact:stability}
    If $M_g = \sumspheres{d}$ and $d \neq 2$, the map $\taut_{\delta, d}$ is an isomorphism in the range of cohomological degrees $* \leq (g-4)/2$. Thus, the map $\taut'_d$ is injective in the same range of degrees.

    For $d=1$, the range of degrees can be improved to $* \leq 2g/3$.
\end{fact}

In particular, the ring $H^*(\BDiff(\sumspheres{d}, D^{2d}); \QQ)$ satisfies \defterm{homological stability}: it is independent of $g$ in a range of cohomological degrees.  Theorem~\ref{thm:lowest-relation-prime} implies that this range of cohomological degrees cannot be improved beyond $* \leq 2g+1$.  

 In Appendix~\ref{sec:low-pontryagins}, we prove another version of Theorem~\ref{thm:finite-generation}. 
\begin{namedtheorem*}{Theorem~\ref{thm:finite-generation-bdry}}
    The image of $\taut_{\delta, d}$ is finitely-generated as a $\QQ$-algebra when $d$ is odd and $g > 1$.
\end{namedtheorem*}

\sectwo{Comparison with known results for surface bundles} \label{sec:surface-case}
In the $d=1$ case, the fiber of our bundle is an oriented genus-$g$ surface $\Sigma_g = M_g^2 = \sumspheres{1}$ and the generalized Miller-Morita-Mumford classes correspond to the classical ones, with $\kappa_i = \gkappa{e^{i+1}} \in H^{2i}\left(\BDiff (\Sigma_g, D^2);\QQ \right)$. 
    The map $\taut_1$ takes the form
\[\taut_1\co \QQ[\kappa_1, \kappa_2, \ldots] \to H^*\left(\BDiff \Sigma_g;\QQ \right).\]

The ring of characteristic classes of surface bundles in rational cohomology coincides with the cohomology of the moduli space of Riemann surfaces $\cM_g$ since 
\[H^*(\BDiff \Sigma_g; \QQ) = H^*(B\Gamma_g; \QQ) = H^*(\cM_g; \QQ)\]
where $\Gamma_g$ is the orientation-preserving mapping class group. (The first equality follows from the theorem of Earle and Eells~\cite{earle-eells69}, which implies that the natural group homomorphism $\Diff \Sigma_g \to \Gamma_g$ is a homotopy equivalence, and thus the bar constructions are weakly homotopy equivalent. The second is true only in rational cohomology and follows from Teichm\"uller theory, see~\cite[\S 12.6]{farb-margalit11} for an overview).  

The image of $\taut_1$ can therefore be thought of as a subring of $H^*(\cM_g; \QQ)$. This subring coincides with the classical \defterm{tautological ring}, as defined in~\cite{mumford83-enumerative}. Techniques of algebraic geometry and low-dimensional topology (hyperbolic geometry, in particular) have been used to obtain many results about the structure of this ring. For example, since $\cM_g$ is a $(6g-6)$-dimensional orbifold, the image of $\taut_1$ must vanish above that degree, and thus be a finite-dimensional vector space over $\QQ$.

More precise results are known; we list the most relevant ones. The image of the map $\taut_1$ is trivial above degree $2(g-2)$   by a theorem of Looijenga~\cite{looijenga95}, and in degree $2(g-2)$ it is one-dimensional~\cite{faber99-conjectural, looijenga95}. Morita~\cite{morita03-generators} showed that the kernel of $\taut_1$ is non-trivial in degree $2\floor{g/3}+2$. However, $\taut_1$ is an isomorphism in degrees $\leq 2\floor{g/3}$ according to Fact~\ref{fact:stability} together with the fact that the map $f^*\co H^*(\BDiff \Sigma_g; \ZZ) \to H^*(\BDiff (\Sigma_g, D^2); \ZZ)$ is an isomorphism in the same range of degrees~\cite{harer85-stability, boldsen12}. For two conjectural complete descriptions of $\ker \taut_1$, which differ for $g > 23$ but are known to be true for $g \leq 23$, see~\cite{faber99-conjectural, pandharipande-pixton13}.

Since the relations in Theorems~\ref{thm:lowest-relation-prime} and~\ref{thm:finite-generation} have high cohomological degree, they follow from Looijenga's theorem in the $d=1$ case. We provide a new proof for the relations of lower degree obtained by Randal-Williams in~\cite{oscar12-rel-revisited}, including all of the existing relations for $g \leq 5$. It is unclear whether our strengthening of Randal-Williams' methods can result in genuinely new relations in the $d=1$ case.

\sectwo{Outline of the paper}

In Section~\ref{sec:definitions-and-main-result}, we define the generalized Miller-Morita-Mumford classes. We then state the main technical result of the paper and the primary source of our relations, Theorem~\ref{thm:fund-relations}.  We outline its proof and apply it to prove Theorem~\ref{thm:lowest-relation-prime}.

The details of the proof of Theorem~\ref{thm:fund-relations} take up sections~\ref{sec:spectral-seq-argument} and~\ref{sec:remainder-of-proof}. In the special case of surface bundles, this work leads to a stronger statement and a new proof of a result of Morita~\cite[Section 3]{morita89-families1}. 

In section~\ref{sec:using-big-theorem}, we use Theorem~\ref{thm:fund-relations} to prove Theorem~\ref{thm:finite-generation} and our other results.  These calculations use methods Randal-Williams developed for surface bundles in~\cite{oscar12-rel-revisited}, originally based on Morita's result. 

Appendix~\ref{sec:low-pontryagins} discusses the relationship between the maps $\taut_d, \taut'_d$ and $\taut_{\delta, d}$, and proves Theorem~\ref{thm:finite-generation-bdry}. Appendix~\ref{sec:pushforward-madness} discusses alternative definitions of the pushforward map on cohomology, which is a crucial ingredient in defining the MMM classes.

\sectwo{Acknowledgments}
I am deeply grateful to my thesis advisor, S\o{}ren Galatius, for his constant support and sage advice. I am also very grateful for the numerous conversations with and the insight provided by  Oscar Randal-Williams. This work would not exist without their support. I thank Alexander Kupers and Jeremy Miller for numerous helpful conversations, as well as suggestions about the content of Appendix~\ref{sec:pushforward-madness}. I thank Sam Lichtenstein for his advice on linear algebra that improved section~\ref{sec:alternating-tensor-section}.  I thank the reviewers for their many insightful comments, corrections, and help with the organization of the paper.  

This work was supported by an NSF Graduate Research Fellowship. I was also supported in part by the Danish National Research Foundation and the Centre for Symmetry and Deformation through a Nordic Research Opportunity grant, by the NSF grant DMS-1105058 and, while I was in residence at the MSRI in Berkeley, CA during the Spring 2014 semester, by the NSF Grant No.\ 0932078000.  I am very grateful to the Stanford University Department of Mathematics for all the support during my years as a graduate student. Part of this work was done during my very pleasant visit to the University of Copenhagen, which I thank for its warm hospitality.
  
\secone{Definitions and our main technical result}\label{sec:definitions-and-main-result}
In this section, we give a more precise definition for terms used in the introduction. We then state the main technical result of this paper and give an informal outline of its proof. Finally, we apply it to prove Theorem~\ref{thm:lowest-relation-prime}. 

Let $M$ be an oriented smooth closed connected manifold and $\Diff M$ is the topological group of \emph{orientation-preserving} diffeomorphisms of $M$ with the $C^\infty$ topology.

\begin{definition}\label{def:manifold-bundle}
    By an \defterm{oriented manifold bundle} (or just \defterm{manifold bundle}), we mean a bundle $E \to B$ with fiber $M$ and structure group $\Diff M$.
\end{definition}

\sectwo{Pushforward maps}\label{sec:pushforward-intro}
 For an oriented manifold bundle $\pi\co E \to B$ with fiber $M$, there is a map of abelian groups $\pi_!\co H^{*+ \dim M}(E; \ZZ) \to H^*(B; \ZZ)$ called the \defterm{pushforward map}, also known as the \defterm{umkehr map} or the \defterm{Gysin homomorphism}. Note that when $\dim M \neq 0$, $\pi_!(1) = 0$  because of the change of cohomological degree, and thus $\pi_!$ is not a ring map. We will give its definition (originally from~\cite{borel-hirzebruch59}) in a more general setting in section~\ref{sec:pushforwards-definition}, Definition~\ref{def:ss-pushforward}.  

To give a little substance to this notion, we mention that in the special case when $E$ and $B$ are closed oriented manifolds, the map $\pi_!$ coincides with the composition of Poincar\'e duality in $E$, the  natural map on homology induced by $\pi$, and Poincar\'e duality in $B$. When restricted to de-Rham cohomology, the map coincides with integration along the fiber (these equivalences are discussed in detail in~\cite{boardman69}).

For our present purposes, it is sufficient to recall one non-trivial property of $\pi_!$. The pushforward map is natural in the sense that, if we form a pullback diagram of manifold bundles
\begin{equation} \label{eq:pullback-of-mfld-bundles}
    \xymatrix{
        f^*(E) \ar[d]^{\pi'} \ar[r]^{f'} & E \ar[d]^{\pi} \\
        A \ar[r]^f & B
    }
\end{equation}
then for any $a \in H^*(E)$, we have $f^*\left(\pi_!(a)\right) = \pi'_!\left(f'^*(a)\right)$.

Further properties of the pushforward map are discussed in Section~\ref{sec:further-properties-pushforwards}.

\sectwo{Definition of the Miller-Morita-Mumford classes}\label{sec:definition-mmm-classes}

Let $P \to B$ be the principal $\Diff M$-bundle corresponding to the manifold bundle $E\to B$.  The group $\Diff M$ acts on the total space of the tangent bundle $TM$ as well as on $M$, and the bundle map $TM \to M$ is equivariant with respect to this action. So, the map \[P \times_{\Diff M} TM  \to P \times_{\Diff M} M = E\] can be given the structure of a bundle over $E$ with the same fiber and structure group as the bundle $TM \to M$.

\begin{definition}\label{def:vertical-tangent-bundle}
   The \defterm{vertical tangent bundle} $T_\pi E$ is the vector bundle of rank $\dim M$ over $E$ defined by the above map. 
\end{definition}

\begin{remark}
    In the special case when the bundle map $E \to B$ is a smooth map between smooth manifolds, 
         the vertical tangent bundle coincides with the sub-bundle of $TE$ that is the kernel of the derivative $Df\co TE \to TB$.
\end{remark}

Since we only consider orientation-preserving diffeomorphisms, $T_\pi E$ is an \emph{oriented} vector bundle. Its characteristic classes determine a map $\gamma\co H^*(\BSO{\dim M}; \ZZ) \to H^*(E; \ZZ)$.

\begin{definition} \label{def:kappa-classes} Let $E \to B$ be a manifold bundle with $m$-dimensional fiber and $p \in H^{l+m}(\BSO{m}; \ZZ)$. The corresponding \defterm{generalized Miller-Morita-Mumford class} or \defterm{kappa class} is defined as follows.
\begin{equation*}
\gkappa{p}\pbundle{E}{B} := \pi_!\left(\gamma^*(p)\right) \in H^l(B; \ZZ).
\end{equation*}
\end{definition}

 The kappa classes are natural with respect to pullbacks of bundles because of the naturality property of pushforwards. To be more precise, the following diagram will commute in the context of the pullback diagram~\eqref{eq:pullback-of-mfld-bundles}. 
\begin{equation*} \vcenter{\xymatrix@C+=.5in{
H^{*+m} (\BSO{m}; \ZZ) \ar[dr]_(.4){\scriptstyle p \mapsto \gkappa{p}{\scriptscriptstyle \pbundle{f^*(E)}{A}}} \ar[r]^-{\scriptstyle p \mapsto \gkappa{p}{\scriptscriptstyle \pbundle{E}{B}}} &
H^*(B; \ZZ)  \ar[d]^{f^*}\\
& H^*(A; \ZZ)
}}
\end{equation*}

Every manifold bundle is a pullback of the universal bundle over $\BDiff M$. So, the kappa classes for any bundle are pullbacks of universal classes $\gkappa{p} \in H^*\left( \BDiff M ; \ZZ\right)$.  Similarly, for $p \in H^{*+m} (\BSO{m}; \QQ)$ there are classes  $\gkappa{p}\pbundle{E}{B} \in H^*(B; \QQ)$ and $\gkappa{p} \in H^*\left( \BDiff M ; \QQ \right)$.

\sectwo{Key source of the relations}
Let us state the main technical result that underlies the relations discussed in this paper. We will give an informal outline of the proof at the end of this section and postpone all details to sections~\ref{sec:beginning-of-proof} and~\ref{sec:end-of-proof}.

We will consider bundles with fiber in the following class of manifolds. This class includes the connected sum of $g$ copies of $S^d \times S^d$. It also includes, for example, connected sums of a space $Q$, which is the total spaces of a bundle with fiber $S'$ and base space $S''$, where $S'$ and $S''$ are smooth homotopy $d$-spheres.

\begin{definition}\label{def:highly-connected-manifold}
    By a \defterm{highly-connected manifold of genus $g$}, we mean a $2d$-dimensional $(d-1)$-connected smooth oriented closed manifold with middle cohomology isomorphic to $\ZZ^{2g}$. Throughout the paper, $M_g$ represents such a manifold.
\end{definition}

\begin{remark}
 If $M$ is an oriented closed smooth $2d$-dimensional $(d-1)$-connected manifold, the Universal Coefficient theorem implies that $H^d(M; \ZZ) \cong \Hom(H_d(M), \ZZ)$, which is a free group. Poincar\'e duality and the fact that $d$ is odd imply that the rank of this group must be even.  
 So, $M$ is a highly-connected manifold of genus $g$ for some integer $g$.
\end{remark}

\begin{theorem}
\label{thm:fund-relations}
Let  $d$ be an \emph{odd} natural number and $M_g$ be a $2d$-dimensional highly-connected manifold of genus $g$.
Let $\pi\co E \to B$ be an oriented manifold bundle with fiber $M_g^{2d}$ and let $a, b \in H^*(E; \ZZ)$ be two classes such that $\pi_!(a) = 0$, $\pi_!(b) = 0$, and $\deg(a)$  is \emph{even}. 

Then, the classes $\pi_!(a \cup a) \in H^{2\deg(a) - 2d}(B; \ZZ)$ and $\pi_!(a \cup b) \in H^{\deg(a) + \deg(b) - 2d}(B; \ZZ)$ satisfy the following two relations.
\begin{equation}
  \label{eq:fund-relation-square}
  (2g+1)! \cdot \pi_!(a \cup a)^{g+1} = 0. 
\end{equation}
\begin{equation}
  \label{eq:fund-relation-product}
  (2g+1)! \cdot \pi_!(a \cup b)^{2g+1} = 0.
\end{equation}
(Note the larger power in the second relation.)
\end{theorem}

\begin{remark}
Because of the $(2g+1)!$ factor in the statement, the theorem is most useful to give relations for cohomology with rational coefficients. It is likely that this factor can be improved somewhat. In~\cite[Section 3]{morita89-families1}, Morita proved the relation~\eqref{eq:fund-relation-square} in the special case of $d=1$ and $\deg a = 2$ with a factor of $\frac{(2g+2)!}{2^{g+1}(g+1)!}$ instead of $(2g+1)!$.
\end{remark}

\sectwo{An application: proof of Theorem~\ref{thm:lowest-relation-prime}}\label{sec:proof-lowest-relation}
In this section, we illustrate Theorem~\ref{thm:fund-relations} by proving Theorem~\ref{thm:lowest-relation-prime}  as an application. Further applications of Theorem~\ref{thm:fund-relations} that result in more elaborate relations are discussed in Section~\ref{sec:using-big-theorem}.

\begin{proposition}\label{pro:lowest-relation}
    Suppose $d\neq 1$ is an odd integer. Let $s=\ceil{\frac{d+1}{4}}$, and let $p_s$ be the $4s$-dimensional Pontryagin class. Then, 
    \begin{equation*}
        \gkappa{p_s^2}^{g+1} = 0 \in H^{(2\text{ or }6)(g+1)}(\BDiff M_g;\QQ)  \text{ where } \deg \gkappa{p_s^2} = 
        \begin{cases}
            2 & \text{if } d \equiv 3\pmod 4 \\ 
            6 & \text{if } d \equiv 1\pmod 4.
        \end{cases}
    \end{equation*}
\end{proposition}

\begin{proof}
    Let $d \geq 3$ is odd.  Let $\pi\co E \to (B = \BDiff M_g)$ be the universal manifold bundle with fiber  $M^{2d}_g$. The $4s$-dimensional Pontryagin class of the vertical tangent bundle gives rise to the class $p_s \in H^{4s}(E;\QQ)$. 
    
    Our choice of $s$ insures that, depending on $d$ mod 4, either $4s = d + 1$ or $4s = d + 3$.  Since under our assumptions $4s < 2d$, we have $\pi_!(p_s) = 0$. Also, $\deg p_s$ is even. Thus, we can apply Theorem~\ref{thm:fund-relations} to obtain the following relation concerning the class $\pi_!\left(p_s^2\right)$, which is either 2- or 6-dimensional.
\[(2g+1)!\pi_!\left(p_s^2\right)^{g+1} = 0 \in H^{(2\text{ or }6)(g+1)}(B; \QQ).\]

    The class $\pi_!\left(p_s^2\right)$ coincides with the class $\gkappa{p_s^2} \in H^{2\text{ or }6}(\BDiff M_g; \QQ)$ by definition. So, rationally $\gkappa{p_s^2}^{g+1} = 0 \in H^{(2\text{ or }6)(g+1)}(\BDiff M_g;\QQ)$ as desired. 
\end{proof}

Since, in the terminology of the introduction, $p_s^2 \in \tbasiss$, Proposition~\ref{pro:lowest-relation} immediately implies Theorem~\ref{thm:lowest-relation-prime} when $d \neq 1$.

Note that Fact~\ref{fact:stability} implies that the class $\gkappa{p_s^2} \in H^*\!\left(\BDiff\left(\sumspheres{d}, D^{2d}\right);\QQ\right)$ is not zero when $g$ is large enough. So, $\gkappa{p_s^2} \neq 0 \in H^*\!\left(\BDiff \left(\sumspheres{d}\right);\QQ\right)$ as well, even though we just showed that $\gkappa{p_s^2}^{g+1} = 0$.

\medskip
When $d=1$ and $g>1$, Theorem~\ref{thm:lowest-relation-prime} follows from Corollary~\ref{cor:kappa-is-decompose}, which in this case is due to Morita (Looijenga's theorem~\cite{looijenga95} is even stronger). The $S^1 \times S^1$ case can be done by replacing $p_s$ with the class $e\gkappa{e^2}$ in the above proof. The $g=0$ case follows from the fact that $\BDiff S^2 \simeq \BSO{3}$ by a theorem of Smale.

\sectwo{Outline of the proof of Theorem~\ref{thm:fund-relations}}\label{sec:overview-proof-disbled}
We aim to prove that a certain power of the class $\pi_!(a \cup b)$ is torsion. If we wanted to prove that $2 \alpha^2 = 0$ for some integral cohomology class $\alpha$, it would be sufficient to decompose it as product of a integral cohomology class of \emph{odd} degree $\beta$  and another class: $\alpha = \beta \cup \gamma$. Our proof is loosely analogous.

In Section~\ref{sec:spectral-seq-argument}, we will use the Serre spectral sequence for the fibration $\pi\co E \to B$ to define the pushforward map on cohomology $\pi_!$. The key result of Section~\ref{sec:spectral-seq-argument} is that, under the assumptions of Theorem~\ref{thm:fund-relations}, the cohomology class $\pi_!(a \cup b)$ is the product of two terms on the $E_2$ page of the spectral sequence, at least one of which -- we call it $\iota$ -- has odd degree (Proposition~\ref{pro:decompose-pushforward-e2}).

The class $\iota$ turns out to be a cohomology class with a $2g$-dimensional, twisted coefficient system. In Section~\ref{sec:remainder-of-proof}, we prove Proposition~\ref{pro:twisted-powers-zero} which implies that since $\deg \iota$ is odd, $\iota^{2g+1}$ is torsion.  We then relate various notions of cup product to conclude that $\pi_!(a \cup b)^{2g+1}$ and $\pi_!(a \cup a)^{g+1}$ are both torsion.

\secone{Spectral sequence argument}\label{sec:spectral-seq-argument}\label{sec:beginning-of-proof}
 In this section, we begin the detailed proof of Theorem~\ref{thm:fund-relations}. A reader more interested in applications might want to skip directly to Section~\ref{sec:using-big-theorem}.

The proof of Theorem~\ref{thm:fund-relations} is most naturally stated in the setting of \defterm{oriented Serre fibrations}. This setting is more general than the setting of manifold bundles. We first define the pushforward map in this generality. Then, our goal is to prove Proposition~\ref{pro:decompose-pushforward-e2}, which in certain cases allows us to decompose cohomology classes of the form $\pi_!(a \cup b)$. 

\sectwo{Oriented Serre fibrations and twisted coefficient systems}
By a \defterm{twisted coefficient system} over $B$, we will mean a bundle of abelian groups over $B$ with some fiber $A$ and the discrete group $\Aut A$, as its structure group. Given a basepoint $* \in B$, twisted coefficient systems correspond bijectively to $\ZZ\left[\pi_1(B,*)\right]$-modules (see e.g.~\cite[Section 5.3]{mccleary01-users-guide}). Moreover, maps and tensor products of twisted coefficient systems correspond to maps and tensor products of $\ZZ\left[\pi_i(B,*)\right]$-modules, respectively.

 Let $E \to B$ be a Serre fibration, $* \in B$ be a chosen basepoint, and $M$ be the homotopy fiber at the basepoint. The homotopy-lifting property of Serre fibrations gives rise to an action of $\pi_1(B,*)$ on the cohomology groups $H^i(M; \ZZ)$ for all $i$. This gives rise to a twisted coefficient system that we denote $\tH^i(M)$. The cup product on cohomology $H^i(M; \ZZ) \otimes H^j(M; \ZZ)\to H^{i + j}(M; \ZZ)$ is a map of $\ZZ\left[\pi_i(B,*)\right]$-modules. So, there is a well-defined cup product on twisted coefficient systems:
\begin{equation}\label{eq:cup-prod-twisted-coeff-collapse}
    \cup\co \tH^i(M) \otimes \tH^j(M)\to \tH^{i + j}(M).
\end{equation}

We are interested in the case where the homotopy fiber is a closed, connected manifold $M^{2d}$. An \defterm{orientation} for such a Serre fibration $E \to B$ is a choice of a trivialization for the twisted coefficient system corresponding to the top cohomology, i.e. a choice of an isomorphism $or\co \tH^{2d}(M) \toiso \ZZ$, where the right-hand side is the untwisted coefficient system over $B$. An \defterm{oriented Serre fibration} is a Serre fibration $E \to B$ that is equipped with a choice of an orientation.

\begin{example}
    Any (oriented) manifold bundle in the sense of Section~\ref{sec:definitions-and-main-result} is an example of an oriented Serre fibration, since the structure group of the manifold bundle preserves the given orientation of the fiber $M$.
\end{example}

\sectwo{Convergence of Serre spectral sequences}
In this section, we recall the features of the convergence theorem for the cohomological Serre spectral sequence that we will need. 

As we will discuss in more detail in Section~\ref{sec:facts-cup-product}, for any coefficient systems $\tA$ and $\tB$ over $B$, there is a notion of \defterm{cohomology with twisted coefficients} and a \defterm{cup product} (different from the one defined in~\eqref{eq:cup-prod-twisted-coeff-collapse})
\begin{equation}\label{eq:cup-prod-twisted-coeff-expand}
    \cup\co\ H^p(B; \tA) \otimes H^q(B; \tB) \to H^{p+q}(B; \tA \otimes \tB).
\end{equation}
Moreover, any map of coefficient systems $f\co \tA \to \tB$ determines a map on cohomology that we will denote $\coeff{f}\co H^*(B; \tA) \to H^*(B; \tB).$

 The Serre spectral sequence for a Serre fibration $\pi\co E \to B$ with fiber $M$ (which, for the purposes of the convergence theorem, can be any CW complex) relates the following two objects:

\begin{enumerate}
\item The cohomology of the total space $H^*(E; \ZZ)$ together with the cup product and a filtration
 \begin{equation}\label{eq:filtration}
   H^*(E; \ZZ) = \cdots = F^{-1} = F^0 H^*(E; \ZZ) \supset F^1 H^*(E; \ZZ) \supset \cdots
\end{equation}
defined as follows. Let $B^{(j)}$ denote the $j$-skeleton of the CW-complex $B$, $J^{(j)} = \pi^{-1}\left( B^{(j)} \right) \subset E$ and $J^{(-1)}=\emptyset$. We set
        \begin{equation*}
            F^iH^*(E) := \ker\left( H^*(E) \to H^*\left( J^{(i-1)} \right) \right) = \image\left( H^*\left( E,J^{(i-1)} \right) \to H^*(E) \right).
        \end{equation*}

Note that this filtration respects the cup product, i.e. the cup product restricts to a map $F^pH^*(E; \ZZ) \otimes F^{p'}H^*(E; \ZZ) \to F^{p+p'}H^*(E; \ZZ)$.

\item The $E_2$ page of the spectral sequence which is the bi-graded ring
    \[E_2^{p,q} := H^p\left(B;\tH^q(M)\right)\] 
with the product specified by the following composition of maps.
 \begin{multline}\label{eq:e2-product-general}
     \bullet\co \ E_2^{p,q} \otimes E_2^{p',q'} =
     H^p\left(B;\tH^{q\vphantom{q'}}(M)\right) \otimes H^{p'}\left(B;\tH^{q'}(M)\right)  \substack{\cup\\ \xrightarrow{\hspace*{0.8cm}} \\ \eqref{eq:cup-prod-twisted-coeff-expand}}  \\
     \hspace*{1.2cm}H^{p+p'}\left(B;\tH^p(M) \otimes \tH^{q'}(M)\right)
     \substack{\coeff{\cup} \\ \xrightarrow{\hspace*{0.8cm}} \\ \eqref{eq:cup-prod-twisted-coeff-collapse}}
     H^{p+p'}\left(B;\tH^{q+q'}(M)\right) = E_2^{p+p',q+q'}
\end{multline}

\end{enumerate}

The convergence theorem relates these two objects by way of the $E_\infty$ page of the spectral sequence:

\begin{theorem}[{Convergence Theorem for the Serre Spectral Sequence,~\cite[Theorem~5.2]{mccleary01-users-guide}}]
\label{thm:convergence-of-serre}
There is a spectral sequence with the $E_2$ page as described above such that the following two definitions of its $E_\infty$ page are equivalent (together with the product structure):
\begin{enumerate}[(a)]
  \item Successive quotients of the filtration~\eqref{eq:filtration} together with the cup product
      \[E_\infty^{p,q} \cong F^pH^{p+q}(E; \ZZ)/F^{p+1}H^{p+q}(E; \ZZ).\]
  \item A sub-quotient of the $E_2$ page obtained by repeatedly taking homology using the differentials in the spectral sequence. Repeatedly taking sub-quotients of a group results in a sub-quotient, so there are subgroups $B^{p,q} \subset Z^{p,q} \subset E_2^{p,q}$ such that 
  \[E_\infty^{p,q} = Z^{p,q} / B^{p,q}.\]
         By a small abuse of language, we write: $Z^{p,q} = \ker(\text{differentials out of }(p,q)\text{ terms})$ and $B^{p,q} = \image(\text{differentials into }(p,q)\text{ terms})$.

  The product structure is induced from the product structure on the $E_2$ page. This uses the fact that all the differentials respect the product on their respective pages of the spectral sequence.
\end{enumerate}
\end{theorem}

\sectwo{Pushforwards and spectral sequences}\label{sec:pushforwards-definition}
 In this section, we assume that the fiber $M^{2d}$ is a $2d$-dimensional oriented closed connected manifold.  
\begin{figure}[ht]

  \centering
  \def\svgwidth{0.98\textwidth}    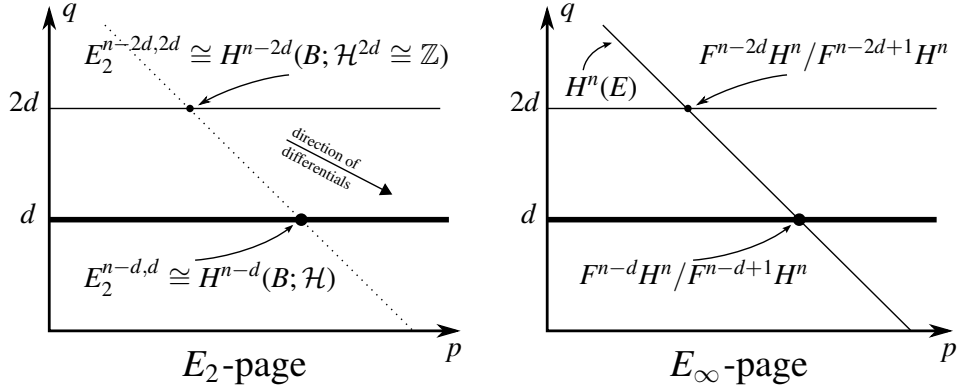 
     \caption{The $E_2$ and $E_\infty$ pages of the Serre Spectral Sequence with fiber a $(d-1)$-connected oriented closed $2d$-dimensional manifold $M$. The entries with total degree $n$ are highlighted. We abbreviate $F^iH^n := F^iH^n\left( E;\ZZ \right)$. }
  \label{fig:serre-ss}
\end{figure}

\begin{lemma}\label{lem:filtrations-jump}
  If $M$ has dimension $2d$, the filtration on cohomology is such that, for all $n$,
  \[ F^{n-2d}H^n(E; \ZZ) = H^n(E; \ZZ). \]
  If $M$ is also $(d-1)$-connected, then we also have
  \[ F^{n-d}H^n(E; \ZZ) = F^{n-2d + 1}H^n(E; \ZZ). \]
\end{lemma}
  (For the indices in this and the following arguments, refer to Figure~\ref{fig:serre-ss})
\begin{proof}
  Since the fiber $M$ is $2d$-dimensional, $E_2^{n - q,q} = 0$ for $q > 2d$, and therefore $0 = E_\infty^{n-q, q} = F^{n-q}H^n(E; \ZZ)/ F^{n - q + 1}H^n(E; \ZZ)$ as well.

  If $M$ is $(d-1)$-connected, then $H^q(M; \ZZ)= 0$ for $2d > q > d$ by Poincar\'e duality. Thus, $E_2^{n - q,q} = 0$ as well in this range.
\end{proof}

From the $E_2$ page onwards, all the differentials in the spectral sequence go in the down-and-right direction. In particular, there are no differentials \emph{into} the $2d$-th row of the spectral sequence (i.e., the $E_i^{n-2d, 2d}$ terms for $i \geq 2$). So, 
\[B^{n-2d, 2d} = \image(\text{differentials into $(n-2d,2d)$ terms}) = 0.\]
The convergence theorem implies that $E_\infty^{n-2d, 2d} \subset E_2^{n-2d, 2d} / B^{n-2d, 2d}$, so
 \begin{lemma}\label{lem:einfty-sub-e2-on-2d-row} 
  We have $E_\infty^{n-2d, 2d} \subset E_2^{n-2d, 2d}$.
\end{lemma}

By definition, $E_\infty^{n-2d, 2d} = F^{n-2d}H^n(E; \ZZ) / F^{n-2d+1}H^n(E; \ZZ)$. We can now state the definition of the pushforward map that we use throughout this paper:

\begin{definition}[{\cite[\S 8]{borel-hirzebruch59}}]\label{def:ss-pushforward}
    If the Serre fibration $\pi\co E \to B$ with fiber $M^{2d}$ is \emph{oriented}, we define the \defterm{pushforward map on cohomology} $\pi_!\co H^*(E; \ZZ)\to H^{*-2d}(B; \ZZ)$ to be the composition of maps
\begin{equation}\label{eq:ss-pushforward} 
 \xymatrix@C-3mm{
H^n(E; \ZZ) \ar@{=}[r] \ar@/_2em/[rrrr]^{\pi_!}& F^{n-2d}H^n(E; \ZZ) \ar@{->>}[r] & E_\infty^{n-2d,2d} \ar@{^(->}[r] & 
E_2^{n-2d,2d} \ar[r]^-{\sim}_{\hspace{-2mm}\coeff{or}}&H^{n-2d}(B; \mathbb{Z}). 
} 
\end{equation}
  \end{definition}

Various properties of the pushforward map (which are not used in this section nor in Section~\ref{sec:remainder-of-proof}) are discussed in sections~\ref{sec:pushforward-intro} and~\ref{sec:further-properties-pushforwards}.

\sectwo{Secondary pushforwards and the decomposition of pushforwards}
Let us now assume that our Serre fibration is oriented and that the fiber $M$ is a $(d-1)$-connected $2d$-dimensional oriented closed manifold. Let us consider the kernel of the map $\pi_!$ we just defined.  

\begin{lemma}
Let $\K{n} := (\ker \pi_!) \cap H^n(E; \ZZ) \subset H^*(E; \ZZ)$. If $M$ is $2d$-dimensional and $(d-1)$-connected, then
\[\K{n} = F^{n-d}H^n(E; \ZZ). \]
\end{lemma}
\begin{proof}
By examining the map~\eqref{eq:ss-pushforward}, we see that the quotient map 
\[ H^n(E; \ZZ) = F^{n-2d}H^n(E; \ZZ) \twoheadrightarrow E_\infty^{n-2d, 2d} = F^{n-2d}H^n(E; \ZZ) / F^{n-2d+1}H^n(E; \ZZ) \]
must take $\K{n}$ to zero and therefore $\K{n} = F^{n-2d+1}H^n(E; \ZZ)$.  Lemma~\ref{lem:filtrations-jump} states that since $M$ is $(d-1)$-connected, $F^{n-2d+1}H^n(E; \ZZ) = F^{n-d}H^n(E; \ZZ)$.
\end{proof}

We will now attempt to repeat the construction of the map~\eqref{eq:ss-pushforward}. The lemma gives us a quotient map $\K{n} = F^{n-d}H^n(E; \ZZ) \twoheadrightarrow E^{n-d,d}_\infty$ (see also Figure~\ref{fig:serre-ss} for indices). It is no longer necessarily true that $E^{n-d,d}_\infty$ is a subset of $E^{n-d,d}_2$, but the convergence theorem states that it is in general a subset of a quotient:
\[ E_\infty^{p,q} = \frac{Z^{p,q}}{B^{p,q}} \subset \frac{E_2^{p,q}}{B^{p,q}}. \]

So, we have the following sequence of maps:
\begin{equation} \label{eq:maps-with-subquotient}
 \xymatrix@C-2mm{
\K{n} \ar@{=}[r] \ar@{-->}[rrrd]^{\xi}&  F^{n-d} H^{n}(E; \ZZ) \ar@{->>}[r] & 
E_\infty^{n-d,d} \ar@{^(->}[r] & 
 \frac{E_2^{n-d,d}}{B^{p,d}} \\
 & & & E_2^{n-d,d} \ar@{->>}[u] \ar@{=}[r] & 
H^{n-d}(B; \tH^d).
} 
\end{equation}
We use the fact that the wrong-way map in the above diagram is surjective to make the following definition:
\begin{definition}
For each $a \in \K{n}$, we define its \defterm{secondary pushforward} $\xi(a) \in E_2^{n-d,d} = H^{n-d}(B; \tH^d)$ to be some element that maps to the same element of $ \frac{E_2^{n-d,d}}{B^{p,d}}$ as $a$ under the maps in~\eqref{eq:maps-with-subquotient}. From now on, we assume that we have fixed a choice of such a $\xi(a)$ for every $a$.

Since there is no reason for $\xi\co \K{n} \dashrightarrow H^{n-d}(B; \tH^d)$ to be a group homomorphism, we will call it a \emph{correspondence} rather than a \emph{map} and denote it with a dashed arrow. 
\end{definition}

\begin{proposition}\label{pro:decompose-pushforward-e2}
Let $a \in \K{p+d}$ and $b \in \K{p'+d}$. The cohomology class $\pi_!(a \cup b) \in H^{p + p'}(B; \mathbb{Z})$ is the image of $\xi(a) \otimes \xi(b)$ under the following map.
\begin{equation}
  \label{eq:decompose-pushforward-e2-version}\xymatrix@R=2.5ex@C+2ex{
   E_2^{p,d} \otimes E_2^{p', d} \ar[r]^-{\bullet} & E_2^{p+p', 2d} \ar[r]^-{\sim}_-{\coeff{or}}  & H^{p + p'}(B; \mathbb{Z})
    \\
         \xi(a) \otimes \xi(b) \ar@{|->}[rr] \ar@{}[u]|{\rotatebox{90}{$\in$}}&  & \pi_!(a \cup b) \ar@{}[u]|{\rotatebox{90}{$\in$}}
}
\end{equation}
\end{proposition}

\begin{proof}
Since the Serre spectral sequence is multiplicative, every term in the diagram~\eqref{eq:maps-with-subquotient} is a subset of some ring. The following diagram combines the multiplication maps on every term.
\[\xymatrix@C+=1in@R=.2in{
 {\hspace{-.5in}\scriptscriptstyle (a \otimes b) \in \,\,} \K{p+d} \otimes \K{p'+d} \ar@{=}[d] \ar[r]^-{\cup}& 
H^{p+p'+2d}(E; \ZZ) \ar@{=}[d]  \ar `r/20pt[d] `[dddd]^{\pi_!}_{\text{(a)}} [dddd]  \\
 F^{p}H^{p+d}(E; \ZZ) \otimes F^{p'}H^{p'+d}(E; \ZZ)  \ar@{->>}[d] &
  F^{p + p'}H^{p + p' + 2d}(E; \ZZ) \ar@{->>}[d] 
\\
 E_\infty^{p,d} \otimes E_\infty^{p',d} \ar[r]^{E_\infty\text{ mult.}}
\ar@{^(->}[d] &
E_\infty^{p + p', 2d} \ar@{^(->}[d]
\\
 \frac{ E_2^{p,d} }{ B^{p,d} } \otimes \frac{ E_2^{p',d} }{ B^{p',d} } 
\ar[r]^{E_2\text{ mult.}}_-{\text{(b)}}&
 \frac{ E_2^{p+p',2d} }{ B^{p+p', 2d} } = E_2^{p+p', 2d} \ar[d]_{\coeff{or}}^{\rotatebox{90}{$\sim$}}
\\
 E_2^{p,d} \otimes E_2^{p',d} \ar@{->>}[u]
\ar@{}[r]|{\hspace{1.2in}\pi_!(a \cup b) \ \in} & 
H^{p+p'}(B; \mathbb{Z}) 
\ar@{}[l]|{\ni\  \xi(a) \otimes \xi(b)  \hspace{1.1in}} 
}\]

We observe the following:
\begin{itemize}
\item The convergence theorem implies that the diagram commutes and the map (b) is well-defined.
\item The composition of maps (a) coincides with the map~\eqref{eq:ss-pushforward} from the definition of $\pi_!$.
\item In the image of the map (b), the group $B^{p+p',2d}$ is zero as we discussed in the proof of Lemma~\ref{lem:einfty-sub-e2-on-2d-row}.
\item The composition of maps from $E_2^{p,d} \otimes E_2^{p',d}$ to $H^{p+p'}(B; \ZZ)$ in the diagram is precisely the map~\eqref{eq:decompose-pushforward-e2-version}.
\end{itemize}

By construction of the secondary pushforward, the image of $\xi(a) \otimes \xi(b)$ in $H^{p+p'}(B;\ZZ)$ is the same as the image of $a \otimes b$, which is precisely $\pi_!(a \cup b)$. 
\end{proof}

\secone{Remainder of the proof of Theorem~\ref{thm:fund-relations}}\label{sec:remainder-of-proof}\label{sec:end-of-proof}
The first goal of this section is to prove the following property of the cup product~\eqref{eq:cup-prod-twisted-coeff-expand}:

\begin{proposition}\label{pro:twisted-powers-zero}
Let $\tH$ be a twisted coefficient system with fiber $\ZZ^k$ with $k \leq 2g$. Let $\iota \in H^*(B; \tH)$ have odd degree. Then,
\[ (2g+1)! \cdot \iota^{2g+1} = 0 \in H^{(2g+1)\deg(\iota)}(B;\tH^{\otimes 2g+1}). \]
\end{proposition}

This proposition is a generalization of the fact that if $\beta \in H^*(B;\ZZ)$ has odd degree, then $2\beta^2 = 0$. Similarly to that fact, the proof relies on the generalized commutativity of cup product with twisted coefficients.

Once we prove Proposition~\ref{pro:twisted-powers-zero}, we will relate it with Proposition~\ref{pro:decompose-pushforward-e2} to complete the proof of Theorem~\ref{thm:fund-relations}.

\sectwo{Cup product and twisted coefficients}\label{sec:facts-cup-product}
In this section, we state the formal properties of cup product for cohomology with twisted coefficients that we use. They generalize familiar properties of the usual cup product. See~\cite{steenrod43-local-coefficients} for a reference.

Cohomology with twisted coefficients assigns  a graded abelian group $H^*(X; \tA)$ to the pair $(X, \tA)$ of a space and a twisted coefficient system. Given two coefficient systems $\tA$ and $\tB$ over the same space $X$, the \defterm{cup product with twisted coefficients} we mentioned in~\eqref{eq:cup-prod-twisted-coeff-expand} is a map $\cup\co H^*(X, \tA) \otimes H^*(X, \tB) \to H^*(X, \tA \otimes \tB)$. Also, given a map of coefficient systems $f\co \tA \to \tB$, there is a corresponding map on cohomology $\coeff{f}\co H^*(X; \tA) \to H^*(X; \tB)$.

The following properties of cup products on cohomology with twisted coefficients will be important for us:

\begin{itemize}
  
   \item The cup product is \emph{associative} in the sense that the two possible cup products of three terms $H^*(X, \tA) \otimes H^*(X, \tB) \otimes H^*(X, \tC) \to H^*(X, \tA \otimes \tB \otimes \tC)$ are the same.

  \item The cup product \emph{commutes with change of coefficients} in the following sense:

     Let $f\co \tA \to \tB$ be and $g\co \tC \to \tD$ be maps of coefficient systems (all over the same space $X$). There is a corresponding map $f \otimes g\co \tA \otimes \tC \to \tB \otimes \tD$. The following diagram commutes. 
      \[\xymatrix{
        H^*(X;\tA) \otimes H^*(X;\tC) \ar[rr]^{\coeff{f} \otimes \coeff{g}} \ar[d]_\cup & & H^*(X;\tB) \otimes H^*(X;\tD) \ar[d]_\cup \\
        H^*(X;\tA \otimes \tC) \ar[rr]^{ \coeff{\left(f \otimes g\right)}} & & H^*(X;\tB \otimes \tD)
      }\]

  \item The cup product is \emph{graded-commutative} in the following sense: 
      
    Let $\tau\co \tA \otimes \tB \to \tB \otimes \tA$ be the map that swaps the coordinates. For $a \in H^p(X; \tA)$ and $b \in H^q(X; \tB)$, we have
    \begin{equation}
      \alpha \cup \beta = (-1)^{pq} \coeff{\tau} (\beta \cup \alpha). \label{eq:twisted-cup-commut}
     \end{equation}

\end{itemize}

 These facts can be proven in the same way as the corresponding facts for the regular cup product; we refer to~\cite[\S 11]{steenrod43-local-coefficients} for details. As in the regular case, graded commutativity of the cup product doesn't hold in general on the level of chains. 

\sectwo{Powers of odd classes and proof of Proposition~\ref{pro:twisted-powers-zero}}\label{sec:alternating-tensor-section}
Before proving Proposition~\ref{pro:twisted-powers-zero}, we need to state two lemmas. 

For any representation $V$ of the symmetric group $S_n$, we denote by $\Alt{V}$ the \defterm{alternating sub-representation}
\[\Alt{V} = \left\{v \in V \mid \forall \sigma \in S_n,\ \sigma \cdot v = sgn(\sigma) v \right\} \subset V. \]

Let $\tH$ be a twisted coefficient system. For any $t$, $\tH^{\otimes t}$ is an $S_t$-representation with the action defined by 
$ \sigma \cdot \left( h_1 \otimes \cdots \otimes h_t \right) = (h_{\sigma(1)} \otimes \cdots \otimes h_{\sigma(t)})$. This action on coefficients also makes the cohomology $H^*(B; \tH^{\otimes t})$ into an $S_t$-representation.

\begin{lemma}
  If $\iota \in H^{\deg(\iota)}(B; \tH)$ with $\deg(\iota)$ odd, then $\iota^t \in \Alt{H^*(B; \tH^{\otimes t})}$.
\end{lemma}
\begin{proof}
  First, consider the $t=2$ case. Since $\iota$ has odd degree, the formula for commutativity of cup product states that, if $\tau \in S_2$ is the non-trivial transposition,
  \[ \coeff{\tau}(\iota \cup \iota) = -\iota \cup \iota = sgn(\tau) \cdot (\iota \cup \iota) \in H^{2\cdot \deg(\iota)}(B; \tH^{\otimes 2}). \]
 
   The general case follows from the facts that any permutation $\sigma \in S_t$ can be decomposed into a product of transpositions, and that the number of these transpositions $\operatorname{mod} 2$ is determined by $sgn(\sigma)$.
\end{proof}

The inclusion $i\co \Alt{\tH^{\otimes t}} \hookrightarrow \tH^{\otimes t}$ is a map of coefficient systems, and therefore induces a map on cohomology. If our coefficient system was a $\QQ$-vector space, we would want to prove that all of $\Alt{H^*(B; \tH^{\otimes t}_\QQ)}$ is in the image\footnote{With a little more work, one can show that $\coeff{i}$ induces an isomorphism $H^*(B; \Alt{\tH^{\otimes t}_\QQ}) \toiso \Alt{H^*(B; \tH^{\otimes t}_\QQ)}$.} of the map $\coeff{i}\co H^*(B; \Alt{\tH^{\otimes t}_\QQ}) \to H^*(B; \tH^{\otimes t}_\QQ)$. We prove an integral version of the same statement.
 
\begin{lemma}\label{lem:powers-and-alt}
    Suppose $\alpha \in \Alt{H^{\deg \alpha}(B; \tH^{\otimes t})}$. Then, $t!\alpha$ is contained in the image of the map $\coeff{i}\co H^*(B; \Alt{ \tH^{\otimes t}}) \to H^*(B;  \tH^{\otimes t})$. By abuse of notation, we will denote this fact by $t! \alpha \in H^*(B; \Alt{\tH^{\otimes t}})$.
\end{lemma}
\begin{proof}
Consider the map on coefficient systems $p\co \tH^{\otimes t} \to \Alt{\tH^{\otimes t}}$ defined by the formula. 
 \[ (v \in \tH^{\otimes t}) \stackrel{p}{\mapsto} \left( \sum_{\sigma \in S_t} sgn(\sigma) \left(\sigma \cdot v\right) \right)\]
 (it is easy to check that its image indeed lies in $\Alt{\tH^{\otimes t}} \subset \tH^{\otimes t}$). The map on cohomology $\coeff{p}$ has image in $H^*(B; \Alt{\tH^{\otimes t}})$.

 At the same time, if $\alpha \in \Alt{H^{\deg \alpha}(B; \tH^{\otimes t})} \subset H^*(B; \tH^{\otimes t})$, then $\coeff{\sigma} \cdot \alpha = sgn(\sigma)\alpha$, and thus
 \[ \coeff{p}(\alpha) = \sum_{\sigma \in S_t} sgn(\sigma) (\coeff{\sigma} \cdot \alpha) = \sum_{\sigma \in S_t} sgn(\sigma)^2 (\alpha) = t! \cdot \alpha. \]

So, $t! \cdot \alpha \in H^*(B; \Alt{\tH^{\otimes t}})$ as desired.
\end{proof}

\begin{proof}[Proof of Proposition~\ref{pro:twisted-powers-zero}]\label{proof-of-twisted-powers-zero}
Let $\iota \in H^*(B; \tH)$ have odd degree and suppose that the twisted coefficient system $\tH$ has a free abelian group of rank $\leq 2g$ as fiber. Then, we have $\Alt{\tH^{\otimes 2g+1}} = 0$. By the above two lemmas, $t! \iota^t \in H^*\left( B; \Alt{\tH^{\otimes t}} \right)$. So, $(2g+1)! \iota^{2g+1} = 0$ as desired.
\end{proof}

\begin{remark}
    In the above proof, the full strength of the assumption that $\tH$ is free abelian is unnecessary. If the fiber of $\tH$ is any finitely generated abelian group such that $\dim_\QQ(\tH \otimes \QQ) \leq 2g$, then $\Alt{\tH^{\otimes 2g+1}}$ will be a torsion group, and so $\iota^{2g+1}$ will be torsion. If $\tH$ is generated by $2g$ elements and has no 2-torsion, $\Alt{\tH^{\otimes 2g+1}}=0$. 
\end{remark}

 \sectwo{Proof of Theorem~\ref{thm:fund-relations}}\label{sec:proof-main-theorem}
 Let  $d$ be an \emph{odd} natural number and $\pi\co E \to B$ be an oriented Serre fibration with fiber $M_g^{2d}$, a $2d$-dimensional highly connected manifold of genus $g$.

\begin{remark}\label{rem:general-bundles1}
    The result we prove is more general than the statement of Theorem~\ref{thm:fund-relations}, as we do not need to make any assumptions about smoothness of the bundle or of $M_g$.  
    However, to apply the theorem to more general bundles, one would need to define some sort of ``kappa classes'' as pushforwards of some cohomology classes on the total space.  The results of Ebert and Randal-Williams from~\cite{ebert_rw-generalized-mmm2013} show that this is possible in rational cohomology for topological bundles with fiber $M_g$.  Their results also suggests that some kappa classes can be defined this way for \defterm{block bundles} with structure group $\widetilde {\Diff M_g}$. To apply the full strength of our results, one would need also to define \defterm{intersection classes} (see Definition~\ref{def:taut-ring-marked-points}) in such a way that Lemma~\ref{lem:int-class-properties} holds.

     \end{remark}

Let us restate Proposition~\ref{pro:decompose-pushforward-e2} from the last section in a form that does not involve spectral sequences. Let $\tH$ denote the twisted coefficient system $\tH^d(M_g)$  and $\omega$ denote the map
  \[\omega\co \tH \otimes \tH \stackrel{\cup}{\longrightarrow} \tH^{2d}(M_g) \stackrel{or}{\longrightarrow} \ZZ.\]

\begin{proposition}\label{pro:decompose-pushforward-cup}
Let $a \in H^{\deg(a)}(E)$ and $b \in H^{\deg(b)}(E)$ be two  classes such that $\pi_!(a) = 0$ and $\pi_!(b) = 0$. Then there are $\iota \in H^{\deg(a) - d}(B; \tH)$ and $\kappa \in H^{\deg(b)-d}(B; \tH)$ that depend only on $a,\,b$ (respectively) such that $\pi_!(a \cup b)$ is the image of $\iota \otimes \kappa$ under the composition of maps
\begin{equation}
  \label{eq:decompose-pushforward}\xymatrix@R=2.5ex{
    H^{\deg(a) - d}(B; \tH) \otimes H^{\deg(b) - d}(B; \tH) \ar[r]^-{\cup} &
    H^{i}(B; \tH \otimes \tH) \ar[r]^-{\coeff{\omega}} &
    H^{i}(B; \mathbb{Z})
    \\
          \iota \otimes \kappa \ar@{|->}[rr] \ar@{}[u]|{\rotatebox{90}{$\in$}}& & \pi_!(a \cup b) \ar@{}[u]|{\rotatebox{90}{$\in$}}
}
\end{equation}
where $i = \deg(a) + \deg(b) - 2d$. 
\end{proposition}

\begin{proof}
    The map~\eqref{eq:decompose-pushforward-e2-version} from Proposition~\ref{pro:decompose-pushforward-e2} is the composition of the product on the $E_2$ page of the spectral sequence~\eqref{eq:e2-product-general} with the orientation isomorphism on coefficients:
\begin{multline*}
 (\coeff{or} \circ \bullet)\co E_2^{p,d} \otimes E_2^{p',d} = H^p(B; \tH) \otimes H^{p'}(B; \tH) \stackrel{\cup}{\longrightarrow} \\
 \stackrel{\cup}{\longrightarrow} H^{p+p'}(B; \tH \otimes \tH) \stackrel{\coeff{\cup}}{\longrightarrow} H^{p+p'}(B; \tH^{2d}(M_g)) \stackrel{\coeff{or}}{\longrightarrow} H^{p+p'}(B; \ZZ).
\end{multline*}

The composition of the last two arrows in the above diagram is precisely $\coeff{\omega}$, and thus the maps~\eqref{eq:decompose-pushforward-e2-version} and~\eqref{eq:decompose-pushforward} coincide.
\end{proof}

Note that if $\deg(a)$ is \emph{even} while $d$ is \emph{odd}, then $\deg(\iota)$ will be \emph{odd}.

\bigskip

Now, the following proposition implies that the map~\eqref{eq:decompose-pushforward} commutes with taking further cup products. The point is that one can compute the value of $\pi_!(a \cup b)^{l}$ from the values of $\iota^l$ and $\kappa^l$. More precisely, we have:

\begin{proposition}\label{pro:decomposition-powers-commute}
The following diagram commutes (only up to sign in the top right corner).

 \newcommand{\mymultline}[1]{ {\begin{array}{c} #1 \end{array}} }
\vspace{-0.5cm}
\[\xymatrix{
\ar@{}[r]|-{\substack{(\iota \otimes \kappa) \\ \otimes \\\cdots \\ \otimes \\ (\iota \otimes \kappa)} \in} &
 \left(\!\!
\mymultline{H^{\deg(a) - d}(B; \tH) \\ \otimes \\  H^{\deg(b) - d}(B; \tH)}
\!\!\right)^{\!\!\!\otimes l}
  \ar[d]^-{\cup} \ar[drrr]^{\cup} 
  \ar[rrr]^{\pm \left( \substack{\text{permute coord., }\\ \text{then }\cup \otimes \cup}  \right)}
  & & &
  \mymultline{H^{(\deg(a) - d)\cdot l}\left(B; \tH^{\otimes l}\right) \\ \otimes \\ H^{(\deg(b) - d)\cdot l}\left(B; \tH^{\otimes l}\right)}
  \ar[d]_-{\cup}^-{\substack{\text{then permute} \\ \text{coefficients}}} 
& \ar@{}[l]|-{\ \ \ \ni \ 
    \substack{
        \!\pm(\iota \cup \cdots \cup \iota) \\ \otimes \\ (\kappa \cup \cdots \cup \kappa)
    }
} 
 \\
& H^{i}(B; \tH \otimes \tH)^{\otimes l} 
\ar[rrr]^{\cup}\ar[d]^-{\left(\coeff{\omega}\right) ^ {\otimes l}}
& & &
H^{il}\!\left(B; (\tH \otimes \tH)^{\otimes l}\right)
\ar[d]^-{\coeff{\left(\omega^{\otimes l} \right)}} 
    \\
\ar@{}[r]|-{\pi_!(a \cup b)^{\otimes l} \  \in} &
H^{i}(B; \mathbb{Z})^{\otimes l} \ar[rrr]^{\cup}
& & &
H^{il}\!\left( B;\ZZ^{\otimes l} \cong \ZZ \right)    & \ar@{}[l]|-{\ni \ \pi_!(a \cup b)^l} 
}\]
\end{proposition}
\begin{proof}
  The commutativity of this diagram follows from repeated applications of the associativity of cup product and the fact that cup product commutes with change of coefficients. In the top right corner, we need to also use the commutativity of cup product, which may insert a sign.
\end{proof}

\begin{proof}[Proof of Theorem~\ref{thm:fund-relations}]
    Let $a, b \in H^*(E; \ZZ)$ be two classes such that $\pi_!(a) = 0$, $\pi_!(b) = 0$, and $\deg(a)$  is even. By the Proposition~\ref{pro:decomposition-powers-commute} and the decomposition~\eqref{eq:decompose-pushforward}, we see that there are
\[ \iota \in H^{\deg(a) - d}(B; \tH) \quad \text{and} \quad \kappa \in H^{\deg(b) - d}(B; \tH)  \]
such that $\pi_!(a \cup b)^{2g+1}$ is the image of $\iota^{2g+1} \cup \kappa^{2g+1}$ under some group homomorphism (the composition of the vertical maps on the right side of the diagram in Proposition~\ref{pro:decomposition-powers-commute}). Since $\deg(a)$ is even and $d$ is odd, $\iota$ has odd cohomological degree. Since $\rank \tH = \rank H^d(M_g; \ZZ) = 2g$, Proposition~\ref{pro:twisted-powers-zero} states that $(2g+1)!\cdot \iota^{2g+1} = 0$. This proves that $(2g+1)! \cdot \pi_!(a \cup b)^{2g+1} = 0$.

Similarly, $\pi_!(a \cup a)^{g+1}$ is the image of $\iota^{g+1} \cup \iota^{g+1} = \iota^{2g+1} \cup \iota$ under a group homomorphism. Again,  $(2g+1)! \cdot \iota^{2g+1} = 0$ and thus $(2g+1)! \cdot \pi_!(a \cup a)^{g+1} = 0$.
\end{proof}

\secone[Generating Relations]{Generating Relations using Methods of Randal-Williams}
\label{sec:using-big-theorem}

In this section, we apply Theorem~\ref{thm:fund-relations} to obtain the results claimed in the introduction as well as some additional relations in $\ker \taut_d$.

\sectwo{Further properties of pushforwards}\label{sec:further-properties-pushforwards} 
To do our calculations, we will use the following properties of the pushforward map. 

\begin{proposition}[Properties of the pushforward map] \label{pro:properties-pushforward}
    Let $\pi\co E \to B$ be an oriented Serre fibration with some closed manifold $M$ as fiber. The pushforward map  $\pi_!\co H^{*+\dim(M)}(E; \ZZ) \to H^*(B; \ZZ)$, as defined in Definition~\ref{def:ss-pushforward}, satisfies the following:  

\begin{enumerate}
\item  For any classes $a \in H^*(E; \ZZ)$ and $b \in H^*(B; \ZZ)$, we have
    \[ \pi_!\left(a \cup \pi^*(b)\right) = \pi_!(a) \cup b. \]
This makes the pushforward into a map of $H^*(B; \ZZ)$-modules, and is sometimes called the \defterm{push-pull formula}.

\item As already mentioned in Section~\ref{sec:pushforward-intro}, pushforwards are \emph{natural} with respect to maps $f\co A\to B$. If $\pi'\co f^*(E) \to A$ is the pullback of the fibration $\pi\co E \to B$, then for any $a \in H^*(E; \ZZ)$, we have $f^*\left(\pi_!(a)\right) = \pi'_!\left(f^*(a)\right)$.

\item Suppose both maps $G \stackrel{\pi''}{\to} E \stackrel{\pi}{\to} B$ are oriented Serre fibrations with (possibly different) closed oriented manifolds as fibers. Then, so is the composition $(\pi \circ \pi'')\co G \to B$. Pushforward maps are \emph{functorial} in the sense that $\pi_! \circ \pi''_! = (\pi \circ \pi'')_!$ as maps from the cohomology of $G$ to the cohomology of $B$.
\end{enumerate}
\end{proposition}
For proofs, we refer to~\cite[\S 8]{borel-hirzebruch59}.

We will also need the following well-known fact:
      \begin{lemma}\label{lem:pushforward-euler}
    Let $\pi\co E \to B$ be an oriented manifold bundle such that $B$ is connected and the fiber is a closed connected oriented manifold $M$. Let $e = e\left(T_\pi E \to E\right) \in H^{\dim M}(E; \ZZ)$. Then, $\pi_!(e) = \chi(M) \in H^0(B; \ZZ)$ where $\chi(M) \in \ZZ$ is the Euler characteristic of $M$.
\end{lemma}
\begin{proof}
    First consider the case when $B$ is a point and $E = M$. The vertical tangent bundle then coincides with the tangent bundle of $M$. Its Euler class is $e(TM \to M) = \chi(M) \cdot [M]$, where $[M]$ is the generator of $H^{\dim M}(M; \ZZ)$ determined by the orientation. It follows easily from Definition~\ref{def:ss-pushforward} that $\pi_!([M]) = 1$ and therefore $\pi_!(\chi(M) \cdot [M])= \chi(M) \in H^0(\{*\})$ by the push-pull formula.

    In general, consider the inclusion of a point $\{*\} \hookrightarrow B$. The induced map on $H^0$ is an isomorphism. The desired statement follows from the fact that the Euler class, the vertical tangent bundle, and the pushforward map are all natural with respect to the pullbacks of bundles.
\end{proof}

\begin{remark}
    For manifold bundles, there is a commonly used alternative definition of the pushforward map that uses the Pontryagin-Thom construction (see~\cite{boardman69} or~\cite[\S 4]{becker-gottlieb75}). It coincides with our definition of the pushforward map rationally and, moreover, the two definitions coincide for integral cohomology as long as $B$ is a CW complex of finite type (see Appendix~\ref{sec:pushforward-madness}). We do not know whether the two definitions coincide nor whether Theorem~\ref{thm:fund-relations} applies integrally to the Pontryagin-Thom pushforward more generally, particularly when $B=\BDiff M$.
\end{remark}

\sectwo{Notation and conventions}\label{sec:notation-conventions-for-calc}
For the remainder of this section, we assume that all cohomology has rational coefficients. Thus, we ignore the integral multiple of Theorem~\ref{thm:fund-relations}.

Throughout, $M^{2d}_g$ denotes a $2d$-dimensional highly-connected manifold of genus $g$ (Definition~\ref{def:highly-connected-manifold}). The most important case is when $M_g=\sumspheres{d}$. 

We assume that $2-2g \neq 0$ throughout, and that $2-2g < 0$ in Section~\ref{sec:proof-finite-generation}.  By the \defterm{tautological ring}, we mean the image of the map $\taut_d$. We denote this subring by $\taut^* = \image\left( \taut_d \right) \subset H^*(\BDiff M_g^{2d};\QQ)$.

\sectwo{Direct applications of Theorem~\ref{thm:fund-relations} and the radical}
In this section, we illustrate how one can obtain relations using Theorem~\ref{thm:fund-relations} directly. These calculations can serve as a warm-up for more complicated calculations described in section~\ref{sec:proof-finite-generation}. We prove that the tautological ring modulo nilpotent elements is generated by at most $2d$ elements. 

\begin{example}\label{exple:relation-low-pontryagin}
    Consider a manifold bundle $\pi\co E \to B$ with $2d$-dimensional fiber $M^{2d}_g$ and $d$ odd (for example, the universal bundle). If a Pontryagin class $p_i \in H^{4i}(E)$ satisfies $4i < \dim M_g$ then $\pi_!(p_i) = 0$. So, the argument of Proposition~\ref{pro:lowest-relation}  applies to it and we have the following relation concerning $\gkappa{p_i^2} = \pi_!( p_i^2) \in H^{4i \cdot 2 - 2d}(B)$.
    \[ \left(\gkappa{p_i^2}\right)^{g+1} = 0 \in H^{(8i- 2d) (g+1)}(B) \text{ for } i < \frac{d}{2} = \frac{\dim M}{4}. \]
\end{example}

\begin{example}\label{exple:change-any-coh-class-to-int-to-zero}
    More generally, let $p \in H^{2 \cdot *}(E)$ be any characteristic class of even degree. Assuming that the Euler characteristic $\chi = 2-2g$ is not zero, we can use the Euler class of the vertical tangent bundle $e \in H^{2d}(E)$ to construct the class $a = p - (e / \chi) \cdot \pi^*\left(\pi_!(p)\right) \in H^*(E)$. Because of the push-pull formula (Proposition~\ref{pro:properties-pushforward}) and Lemma~\ref{lem:pushforward-euler}, this class satisfies $\pi_!(a) = 0$. 

Let $q \in H^{2\cdot*}(E)$ be another such class. We apply the procedure just described and Theorem~\ref{thm:fund-relations} to obtain the following formula (we use the notation $\pi_!(p) = \gkappa{p}$).
\begin{multline}\label{eq:relation-p-kappa_q-kappa} 
    0 = \left( \pi_!\left((p - \frac{e}{\chi}\gkappa{p})(q - \frac{e}{\chi}\gkappa{q})\right) \right)^{2g+1} 
       = \\ \left(\gkappa{pq} - \frac{\gkappa{ep}}{\chi}\gkappa{q}-\frac{\gkappa{eq}}{\chi}\gkappa{p} + \frac{\gkappa{e^2}}{\chi^2} \gkappa{p}\gkappa{q} \right)^{2g+1}\hspace{-1.9em}.\hspace{1.9em}
\end{multline}
\end{example}

Let $\sqrt{0} \subset \taut^*$ denote the \defterm{radical} of the tautological ring (that is, the ideal consisting of all the nilpotent element, also known as the \defterm{nilradical}). The following easy fact, together with our finite-generation result (Theorem~\ref{thm:finite-generation}), provides motivation to consider it.
\begin{lemma}\label{lem:radical-and-fg}
If a graded commutative ring $A^*$ is finitely generated as an $A^0$-algebra and $A^0$ is a field, then the following statements are equivalent:
\[ 1)\  A^*\text{ is finite-dimensional}\qquad 2)\ A^* / \sqrt{0} = A^0
\qquad 3)\ \dim_{\textrm{Krull}} A^* = 0. \]
\end{lemma}

Example~\ref{exple:change-any-coh-class-to-int-to-zero} implies

\begin{lemma}\label{lem:radical-ideals-generators}
    In the ring $\taut^*/\sqrt{0}$, $\gkappa{pq}$ is in the ideal generated by $\gkappa{p}$ and $\gkappa{q}$.
\end{lemma}

\begin{proof}
    The expression~\eqref{eq:relation-p-kappa_q-kappa} implies that $\gkappa{pq} - \frac{\gkappa{ep}}{\chi}\gkappa{q}-\frac{\gkappa{eq}}{\chi}\gkappa{p} + \frac{\gkappa{e^2}}{\chi^2} \gkappa{p}\gkappa{q} \in \sqrt{0}$.
\end{proof}

\begin{proposition}\label{pro:taut-mod-radical-generators}
    If $g \neq 1$, the ring $\taut^*/\sqrt{0}$ is generated by the $2d$ elements in the set $E=\left\{ \gkappa{p_i}, \gkappa{p_i \cdot e} \mid 1 \leq i \leq d  \right\}$. So, the Krull dimension of the ring $\taut^*$ is at most $2d$. 
\end{proposition}
\begin{proof}
    Every generator of $\taut^*$ that is not in $E$ can be written as $\gkappa{pq}$ so that $p,q \neq e$. This uses the fact that $p_d = e^2$. It follows that whenever either $\gkappa{p}$ or $\gkappa{q}$ is not zero, it has strictly positive cohomological degree. By Lemma~\ref{lem:radical-ideals-generators},  $\gkappa{pq}$ is decomposable in $\taut^*/\sqrt{0}$ as a polynomial in classes of smaller degree. It follows that $\taut^*/\sqrt{0}$ is generated by the elements of $E$.
\end{proof}

\sectwo{The classifying spaces of manifolds with marked points}
To get additional relations, we will use the methods of~\cite{oscar12-rel-revisited}. Those methods involve certain natural bundles with structure group $\Diff M_g$ and fiber $(M_g)^{\times n} =M_g \times \cdots \times M_g$. In this section, we introduce these bundles and the special characteristic classes they possess. The discussion is completely analogous to the two-dimensional case, as described in~\cite[Section 2.1]{oscar12-rel-revisited}. 

\begin{notation}
     In this section, we denote the universal bundle $\EDiff M_g \times_{\Diff M_g} M_g \to \BDiff M_g$ with fiber $M_g$ as $\Mgtot^{2d} \to \Mg^{2d}$.  The notation refers to the fact that in the $d=1$ case, the space $\Mg^2$ has the same rational cohomology as the moduli space of Riemann surfaces.  We will also use the notation `$\hquot$' for homotopy quotients: $(- \hquot \Diff M) := (- \times_{\Diff M} \EDiff M)$. For example, $\Mg = {*}\hquot\Diff M_g$ and $\Mgtot = M_g \hquot\Diff M_g$.
\end{notation}

\newcommand{\Maps}[2]{\operatorname{Map}\!\left( #1; #2 \right)}
For a finite set $I$, we let $\Maps{I}{M_g}$ be the space of maps $I \to M_g$, 
\[\Mgn{I} := \Maps{I}{M_g} \hquot \Diff M_g ,\text{ and } \Mgn{n} := \Mgn{\{1, \ldots, n\}}.\]
The fiber of the natural map $\Mgn{n} \to \Mg$ is $(M_g)^{\times n}$. So, a map from any space $B$ to $\Mgn{n}$ gives rise to a manifold bundle over $B$ with fiber $M_g$ together with a choice of $n$ ordered points in each fiber.
 
For $J \subset I$, there are natural projections $\pi^I_J\co\Mgn{I} \to \Mgn{J}$ and $ \pi_\emptyset^I\co \Mgn{I} \to \Mg$. We can identify  the bundle $\Mgn{1} \to \Mg$ with the universal bundle $\Mgtot \to \Mg$. More generally, the pullback of the universal bundle $\left(\pi_\emptyset^I\right)^*(\Mgtot)$ and $\Mgn{I \sqcup \{\mpoint\}}$ are canonically isomorphic as bundles over $\Mgn{I}$. 

\begin{definition}\label{def:taut-ring-marked-points}
By the \defterm{tautological subring of the cohomology of $\Mgn{I}$} we mean the subring $\taut^*(\Mgn{I}) \subset H^*(\Mgn{I})$ generated by the following three types of classes that we call the \defterm{fundamental tautological classes}:

\begin{itemize}
\item The generalized MMM classes $\gkappa{c} \in H^*(\Mgn{I})$ that are pulled back from $H^*(\Mg)$ using the canonical map $\Mgn{I} \to \Mg$ (there is one such class for each $c \in H^*(\BSO{2d})$).

\item For each choice of $i \in I$, there is a canonical map $\pi_{i}^I\co \Mgn{I} \to \Mgn{\{i\}} \cong \Mgtot$. The vertical tangent bundle determines a classifying map $\gamma\co \Mgtot \to \BSO{2d}$. For each $c \in H^*(\BSO{2d})$ and $i \in I$, we define the class $\gpsi{c}{i}\in H^*(\Mgn{I})$ as the pullback of $c$ via the composition of the above-mentioned maps\footnote{We use parentheses in the notation to prevent confusion with the notation $p_i$ for the $i$-th Pontryagin class.}. 

Note that given $c,d \in H^*(\BSO{2d})$, we clearly have $\gpsi{(cd)}{i} = \gpsi{c}{i}\gpsi{d}{i}$.

\item For each subset $S \subset I$, we consider the \defterm{intersection class} 
\[\intclass{S} \in H^{2d \cdot \left(|S| - 1\right)}(\Mgn{I})\]  defined below. We will write simply $\intclass{1,2}$ for $\intclass{\{1,2\}}$. 
\end{itemize}
\end{definition}

\begin{definition}
    For $S \subset I$, let $\Maps{I/S}{M_g} \subset \Maps{I}{M_g}$ be those maps that send all elements of $S$ to the same point. Note that this inclusion has codimension $(|S|-1) \cdot \dim M$. Let $\Mgn{I/S} = \Maps{I/S}{M_g} \hquot \Diff M_g$. There is an inclusion $i_S\co\Mgn{I/S} \hookrightarrow \Mgn{I}$. As shown in~\cite[Lemma 2.1]{oscar12-rel-revisited}, this inclusion has a Thom class 
    \[\intclass{S}' \in H^{2d(|S|-1)}\left(\vphantom{\sum} \Mgn{I}\!,\, \Mgn{I} - \Mgn{I/S}; \ZZ\right).\]
We define the \defterm{intersection class $\intclass{S}$} to be the image of $\intclass{S}'$ in $H^*(\Mgn{I})$. 
\end{definition}

\begin{lemma}\label{lem:int-class-properties}
  The classes $\intclass{S}$ satisfy the following: 
  \begin{enumerate}[(i)]
  \item For $S \subset I' \subset I$, the class $\intclass{S} \in H^*(\Mgn{I})$ is a pullback of the corresponding class $\intclass{S} \in H^*(\Mgn{I'})$ via the map $(\pi_{I'}^{I})^*$.
  \item If $S$ and $S'$ intersect at a single point, then $\intclass{S}\intclass{S'} = \intclass{S \cup S'} $. For example, in $\Mgn{\{1,2,\mpoint\}}$, we have: $\intclass{1,\mpoint}\intclass{2,\mpoint} = \intclass{1,\mpoint} \intclass{1, 2}$. 
  \item In $\Mgn{2}$, we have $\intclass{1,2}^2 = \intclass{1,2} \cdot \gpsi{e}{1}$ where $e$ is the Euler class.
  \item For any characteristic class $c$, $\intclass{1,2} \cdot \gpsi{c}{1} = \intclass{1,2} \cdot \gpsi{c}{2}$.
  \item The pushforward of the class $\intclass{1,2} \in H^{2d}(\Mgn{2})$ is 1, i.e.
      \[ \left(\pi_{\{1\}}^{\{1,2\}} \right)_! (\intclass{1,2}) = 1 \in H^0(\Mgn{1}). \]
  \end{enumerate}
\end{lemma}

The proof of this lemma is similar to the arguments in~\cite[\S 11]{milnor74-characteristic}, see also~\cite[Lemma~2.1]{oscar12-rel-revisited}. The proof of part (v) is very similar to the proof of Lemma~\ref{lem:pushforward-euler}.

Our next goal is to be able to compute the pushforward of any tautological class in $H^*(\Mgn{I})$ via the projection maps $\pi^I_J$. We will use the properties of the pushforward described in Section~\ref{sec:further-properties-pushforwards}. 

Lemma~\ref{lem:int-class-properties} and the naturality of the pushforward imply the following.
\begin{lemma}\label{lem:pushforward-int-class}
  For any finite set $I$, we have
  \[ \left(\pi^{I \sqcup \{\mpoint\}}_I\right)_!(\intclass{i\mpoint}) = 1 \text{ and } \left(\pi^{I \sqcup \{\mpoint\}}_I\right)_! (\gpsi{c}{\mpoint}) = \gkappa{c} \]
for all $i \in I$ and $c \in H^*(\BSO{2d})$. We use the convention $\gkappa{e} = \chi = 2-2g$.
\end{lemma}

Furthermore, it is possible to rewrite a tautological class in $H^*(\Mgn{I \sqcup \{\mpoint\}})$ in terms of a tautological classes in $H^*(\Mgn{I})$ as follows:

\begin{lemma}\label{lem:simplify-monomial-taut}
    We can simplify any monomial in the fundamental tautological classes $m \in H^*\left(\Mgn{I \sqcup \{\mpoint\}}\right)$ in one of the following ways:

    \begin{itemize}
        \item If the monomial contains $\intclass{i, \mpoint}$ for some $i \in I$, then it can be rewritten as $m = \intclass{i,\mpoint} \cdot n'$ where $n'$ is a monomial in classes that do not involve the marked point '$\mpoint$'. That is, $n' = \left(\pi^{I \sqcup \{\mpoint\}}_I\right)^*(n)$ where $n$ is a monomial in tautological classes of $\Mgn{I}$. 
        \item Otherwise,  the monomial can be rewritten as $m = \gpsi{c}{\mpoint} \cdot n'$ where $c$ is a  product (possibly empty) of characteristic classes of the vertical tangent bundle and $n'$ is as before. 
    \end{itemize}
\end{lemma}

\begin{proof}
    If $m$ does not contain any $\intclass{i, \mpoint}$s, reordering its terms will put it in the required form. Otherwise, we use the relations $\intclass{i,\mpoint}\intclass{j,\mpoint} = \intclass{i,\mpoint}\intclass{i,j}$ and $\intclass{i,\mpoint}\gpsi{c}{\mpoint} = \intclass{i, \mpoint}\gpsi{c}{i}$ from Lemma~\ref{lem:int-class-properties} to get rid of any classes that involve '$\mpoint$' except for the single $\intclass{i,\mpoint}$.
\end{proof}

The push-pull formula and the above lemmas give us the following procedure to compute the pushforward of a general tautological class:

\begin{procedure}\label{proced:rules-for-integration}
    The result of applying the pushforward map \[\left(\pi^{I \sqcup \{\mpoint\}}_I\right)_!\co H^*(\Mgn{I \sqcup \{\mpoint\}}) \to H^*(\Mgn{I})\] to a tautological class can be computed as follows, one monomial at a time. First, simplify the monomial $m\in H^*(\Mgn{I \sqcup \{\mpoint\}})$ using Lemma~\ref{lem:simplify-monomial-taut}.  Then, apply the push-pull formula and Lemma~\ref{lem:pushforward-int-class} to get:
 \begin{align*}
     \text{If } &m = \intclass{i,\mpoint} \cdot \left(\pi^{I \sqcup \{\mpoint\}}_I\right)^*(n), &&\text{we have } \left(\pi^{I \sqcup \{\mpoint\}}_I\right)_!(m) = \left(\pi^{I \sqcup \{\mpoint\}}_I\right)_!(\intclass{i\mpoint})\cdot n =n. \\ 
     \text{If } &m = \gpsi{c}{\mpoint} \cdot \left(\pi^{I \sqcup \{\mpoint\}}_I\right)^*(n), &&\text{we have } \left(\pi^{I \sqcup \{\mpoint\}}_I\right)_!(m) = \left(\pi^{I \sqcup \{\mpoint\}}_I\right)_!(\gpsi{c}{\mpoint})\cdot n = \gkappa{c} \cdot n.  
 \end{align*}
 
 Note that in the second case above, if we have $\gpsi{c}{\mpoint}=1$, then the pushforward will be zero.
\end{procedure}

\begin{example}
We can compute a pushforward as follows.
\[\left(\pi^{\{i, j, \mpoint\}}_{\{i,j\}}\right)_! \left(\intclass{i,\mpoint}^3\intclass{j,\mpoint}^2\gpsi{d}{\mpoint}\gkappa{e}\right) = \left(\pi^{\{i, j, \mpoint\}}_{\{i,j\}}\right)_!\left(\intclass{i,\mpoint}\gpsi{e}{i}^2\intclass{i,j}^2\gpsi{d}{i}\gkappa{e}\right) = \gpsi{e}{i}^2\intclass{i,j}^2\gpsi{d}{i}\gkappa{e}.\]
\end{example}

Since pushforward maps are functorial, we can apply Procedure~\ref{proced:rules-for-integration} several times to calculate $\left(\pi^I_J\right)_!$ for any $J \subset I$. 
There also exist formulas for calculating $\left(\pi^I_{\emptyset}\right)_!$ of a tautological monomial in $H^*(\Mgn{I})$ in one step. See~\cite[Section 2.7]{oscar12-rel-revisited} for details.

\sectwo{Randal-Williams' method and proof of Theorem~\ref{thm:finite-generation}}\label{sec:proof-finite-generation}
We can obtain numerous relations in the cohomology of $\Mg$ by applying the following idea of~\cite{oscar12-rel-revisited}. 

\begin{procedure}
First, we construct some tautological class $c \in \taut^*\left( \Mgn{I \sqcup \{\mpoint\}} \right)$ such that $\left(\pi^{I \sqcup \{\mpoint\}}_I\right)_!(c) = 0$. Applying Theorem~\ref{thm:fund-relations} to one or two such classes will tell us that some polynomial in the ring $\taut^*\left( \Mgn{I} \right)$ is equal to zero. We may multiply this relation by any other polynomial and apply $\left(\pi^I_\emptyset\right)_!$ to the result to get a relation among the tautological classes of $\Mg$. 
\end{procedure}

We can obtain more relations than were obtained in~\cite{oscar12-rel-revisited} because the version of our Theorem~\ref{thm:fund-relations} that~\cite{oscar12-rel-revisited} used (from~\cite{morita89-families1})   only applies when the cohomological degree of $c$ is 2 and does not allow using two cohomology classes at once.

\begin{example}
  We illustrate this procedure by repeating the following example from~\cite[Section 2.2]{oscar12-rel-revisited} with our notation. Consider the bundle $\pi\co \Mgn{\{1, \mpoint\}} \to \Mgn{1}$ (which has fiber $M_g$). The following class pushes forward to 0:
  \[ \chi \intclass{1\mpoint} - \gpsi{e}{\mpoint} \in H^*(\Mgn{\{1, \mpoint\}}). \]

Theorem~\ref{thm:fund-relations} applies to give us the following relation in the ring $\taut^*(\Mgn{1})$, which we then simplify using Procedure~\ref{proced:rules-for-integration} and related lemmas.
\begin{multline}
  0 = \left(\pi_!\left((\chi \intclass{1\mpoint} - \gpsi{e}{\mpoint})^2 \right)\right)^{g+1} = 
        \left( \pi_! \left( \chi^2\intclass{1\mpoint}\gpsi{e}{1}- 2\chi \intclass{1\mpoint}\gpsi{e}{1}+\gpsi{e}{\mpoint}^2  \right) \right)^{g+1} = \\
     = \left((\chi-2)\chi \gpsi{e}{1} + \gkappa{e^2} \right)^{g+1} = 
       \sum_{i=0}^{g+1} \binom{g+1}{i} ((\chi - 2) \chi \gpsi{e}{1})^i(\gkappa{e^2})^{g+1 - i}.
 \end{multline}

Let us now assume that $\chi = 2-2g < 0$. 
For each integer $k$, we can multiply both sides of the formula by $\frac{\gpsi{e}{1}^k}{((\chi - 2)\chi)^{g+1}}$ and apply $(\pi^{\{1\}}_\emptyset)_!$ to both sides to get the following relation in the cohomology of $\Mg$.
\begin{equation}
    0 = \sum_{i=0}^{g+1} \binom{g+1}{i} \gkappa{e^{i+k}} \left(\frac{\gkappa{e^2}}{(\chi-2)\chi}\right)^{g+1 - i} \in H^{2d(g+k)}(\Mg). 
\end{equation}
(where we should keep in mind that $\gkappa{e^0} = 0$ and $\gkappa{e^1} = \chi$).

\end{example}

\begin{corollary}\label{cor:kappa-is-decompose}
  From the above example, we can see that for  $k \geq 0$, the degree $2d(g+k)$ class $\kappa_{g+k} = \gkappa{e^{k+g+1}}$ can be written as a polynomial in lower kappa classes.
\end{corollary}

\begin{example} \label{exple:fin-gen-relation}
    Assume that $\chi \neq 0$ and fix any $p \in H^{2i}(\BSO{2d})$. We obtain a relation in the cohomology of $\Mgn{1}$ by applying the second part of Theorem~\ref{thm:fund-relations} to the classes $a= \intclass{1\mpoint} -\gpsi{e}{\mpoint}/\chi \in H^{2d}(\Mgn{\{1,\mpoint\}})$ and $b = \gpsi{p}{\mpoint} - (\gpsi{e}{\mpoint} / \chi) \gkappa{p} \in H^{2i}(\Mgn{\{1,\mpoint\}}$ (both classes push down to zero in $\Mgn{1}$). The Theorem gives us the following formula.
\begin{multline} \label{eq:relation-fin-gen}
 0 = \left( \vphantom{\sum^{1}}  \left(\pi^{\{1,\mpoint\}}_{\{1\}} \right)_! \left( \vphantom{\sum} \left(\gpsi{p}{\mpoint}- (\gpsi{e}{\mpoint} / \chi) \gkappa{p}\right)\left(\intclass{1\mpoint} - (\gpsi{e}{\mpoint} / \chi ) \right) \right) \right) ^{2g+1} = \\
= \left( \gpsi{p}{1} - \frac{\gkappa{ep}}{\chi} - \frac{\gpsi{e}{1} \gkappa{p}}{\chi} + \frac{\gkappa{e^2} \gkappa{p}}{\chi^2} \right)^{2g+1} \in H^*\left(\vphantom{\sum} \Mgn{1}\right).
\end{multline}
\end{example}

We will use the above example to prove Theorem~\ref{thm:finite-generation}. First, we need the following corollary.

Let $\cA \subset \taut^*(\Mg)$ be the \defterm{augmentation ideal} generated by all the elements of the tautological subring that have a non-zero cohomological degree, and let $\cD = \cA \cdot \cA$ be the ideal of the decomposable elements. 

\begin{lemma}\label{lem:finite-generation-ideals}
  Assume $g>1$. There is an integer $N > 0$ that depends only on $g$ and $d$ such that for all $p,q \in H^*(\BSO{2d})$ with $\deg p > 0$,
  \[ \gkappa{\left(p^Nq\right)} \in \cD \subset \taut^*(\Mg). \]
\end{lemma}
\begin{proof}
We will use $N = (2d+1)(2g+1)$. If $1 \leq \deg p < 2d$, we replace $p$ with $p^{2d+1}$. This allows us to assume that $\deg p > 2d$.

  Let $\cA', \cB', \cD' \subset \taut^*(\Mgn{1})$ be the following ideals.
 \begin{equation*}
\cA'  = \left(\gkappa{t} \mid  t \in H^{>2d}(\BSO{2d})\right), \ \  
\cB'  = \left(\gpsi{t}{1} \mid  t \in H^{>2d}(\BSO{2d})\right), \ \ 
\cD'  = \cA' \cdot (\cA' + \mathcal{B}'). 
\end{equation*}

We observe that:
  \begin{enumerate}
      \item $\gpsi{p}{1}^{2g+1} \in \cD'$. To see this, note that since $\deg(p) > 2d$, $\gpsi{e}{1}\gkappa{p}$ and $\gkappa{e^2}\gkappa{p}$ are in $\cD'$. Using our assumption that $g>1$, the formula~\eqref{eq:relation-fin-gen} implies that $\gpsi{p}{1}^{2g+1} \in \cD'$ as well.

  \item The pushforward operation $\left(\pi^{\{1\}}_\emptyset\right)_!$ takes $\cD' \subset \Mgn{1}$ into $\cD \subset \Mg$.
  \end{enumerate}
  
   It follows that $\gpsi{p}{1}^{2g+1}\gpsi{q}{1} = \gpsi{(p^{2g+1}q)}{1} \in \cD'$ for all $q \in H^*(\BSO{2d})$ and, therefore, $\gkappa{(p^{2g+1}q)} \in \cD$. 
\end{proof}

Now, we can finally prove that the tautological ring is finitely generated.
\label{sec:sec-where-fin-gen-proved}
\begin{proof}[Proof of Theorem~\ref{thm:finite-generation}]
    The infinitely many elements $\gkappa{\left(e^{a_0}\prod_{i=1}^{d} p_i^{a_i}\right)}$ (where $a_i$-s are non-negative integers and $p_i$-s are the Pontryagin classes) generate the tautological ring rationally. By the previous lemma, there is a constant $N$ such that the elements where any of the $a_i$-s are greater than $N$ are decomposable. In other words, any such generator is expressible as a polynomial in kappa classes of lower cohomological degree.

  So, the finitely many generators of cohomological degree less than $\deg\left( \gkappa{\left(e^{N}\prod_{i=1}^{d} p_i^{N}\right)} \right)$ generate the whole tautological subring of $H^*\left(\BDiff M_g; \QQ \right)$.
\end{proof}

\sectwo{Randal-Williams' calculations and high-dimensional manifolds}\label{sec:all-oscar-relations}

\newcommand{\IRW}{\mathcal{I}_g^\text{RW}}

Randal-Williams obtained numerous examples\footnote{These include all the relations that exist for $d=1$, $g\leq 5$ in degrees $*\leq 2(g-2)$. In higher degrees, the tautological ring vanishes completely according to~\cite{looijenga95}.} of relations in $d=1$ case for $g=3,4,5,6,9$ in~\cite[Section 2]{oscar12-rel-revisited} using computer calculations.  He also produced a more explicit family of relations in every genus in \cite[Section 2.7]{oscar12-rel-revisited}. 

Formally, all the equations and examples from~\cite{oscar12-rel-revisited} can be interpreted as generators for some ideal $\IRW\subset \QQ[\kappa_1, \kappa_2, \ldots]$. In this language, the result of~\cite{oscar12-rel-revisited} is that the ideal $\IRW$ is in the kernel of the map $\QQ[\kappa_1, \kappa_2, \ldots] \to H^*\left(\BDiff M^2_g\right)$ in the $d=1$ case.  We will show the following.
\begin{proposition}\label{pro:all-oscar-relations}
    For all odd $d$, the same ideal $\IRW$ is in the kernel of the corresponding map $\QQ[\kappa_1, \kappa_2, \ldots] \to H^*\left(\BDiff M^{2d}_g\right)$.
\end{proposition}
As we mentioned in the introduction, this is surprising since the cohomological degree of $\kappa_i = \gkappa{e^{i+1}} \in H^{2di}\left(\BDiff M^{2d}_g\right)$ depends on $d$. 

\begin{example}[{\cite[Example 2.5]{oscar12-rel-revisited} and Proposition~\ref{pro:all-oscar-relations}}]\label{exple:oscar-g4-relations}
    For all odd values of $d$ and $g=4$, we have the following relations in $H^*(\BDiff M^{2d}_4)$.
\begin{equation*}3\kappa_1^2 = -32\kappa_2 \in H^{4d}(\BDiff M^{2d}_4) \text{ and } \kappa_2^2=\kappa_1\kappa_2=\kappa_3=0 \in H^{6d}(\BDiff M^{2d}_4).\end{equation*}
\end{example}

For more examples of relations,  see~\cite[Examples~2.3-2.7]{oscar12-rel-revisited}.

\begin{proof}[Proof of Proposition~\ref{pro:all-oscar-relations}] 
First, we repeat the key steps of~\cite{oscar12-rel-revisited} in our level of generality.

\begin{enumerate}
\item 
    Let $M^{2d}_g \to E \overset{\pi}{\to} B$ be a manifold bundle. Let $c \in H^{2d}(E)$ and $q = \pi_!(c) \in H^0(B) \cong \ZZ$. The relation~\eqref{eq:fund-relation-square} from Theorem~\ref{thm:fund-relations} applied to the cohomology class $\frac{\chi \cdot c - q \cdot e}{\gcd (\chi, q)}$ implies that the cohomology class
    \begin{equation}\label{eq:oscar-thmA}
         \Omega(E,c):=\frac{1}{(\gcd (\chi, q))^2}\left(\chi^2 \pi_!(c^2) - 2q\chi \pi_!(e \cdot c) + q^2\kappa_1 \right) \in H^{2d}(B)
     \end{equation}
     has the property that $\Omega(E,c)^{g+1}$ is torsion.

     This is precisely the version of \cite[Theorem A]{oscar12-rel-revisited} that is stated on~\cite[top of p. 1775]{oscar12-rel-revisited} for $d=1$ (we use slightly different notation). Note that the only part of the expression~\eqref{eq:oscar-thmA} that depends on $d$ is the cohomological degree. 

 \item 
     Consider the bundle $M_g \to \Mgtot(n) \to \Mgn{n}$, defined as the pullback of the universal bundle $\Mgtot \to \Mg$ to $\Mgn{n}$. Following~\cite{oscar12-rel-revisited}, our next step is to apply~\eqref{eq:oscar-thmA} to a particular class in the cohomology of its total space.
     
     Recall that $\Mgtot(n) \cong \Mgn{\{1, \ldots, n, \mpoint\}}$. Given a vector $A = (A_1, \ldots, A_n) \in \ZZ^n$, consider the class
     \[ c_A := \sum_{i=1}^n A_i \intclass{i\mpoint} \in H^{2d} \left(\Mgtot(n) \right) = H^{2d} \left( \Mgn{\{1, \ldots, n, \mpoint\}}\right). \]

We define the class $\Omega_A := \Omega\left( \Mgtot(n), c_A\right)$ using~\eqref{eq:oscar-thmA}. It will satisfy $ \Omega_A^{g+1} = 0 \in H^{2d(g+1)}(\Mgn{n};\QQ)$. The expression for this class does not depend on $d$ and coincides with the formula~\cite[(2.1)]{oscar12-rel-revisited}.

 \item 
We can now obtain non-trivial examples of relations as follows, repeating the procedure from~\cite[Section~2.4]{oscar12-rel-revisited}. Take the equation $\Omega_A^{g+1}=0$ for some values of $A$ and $n$, and perhaps multiply it by another tautological class that doesn't involve Pontryagin classes. Then, apply the pushforward $\left(\pi^{\{1, \ldots, n\}}_\emptyset\right)_!$ to the result to obtain an element of the kernel of the map $\QQ[\kappa_i \mid i \in \NN] \to H^*(\Mg^d,\QQ)$.  Every relation obtained in~\cite{oscar12-rel-revisited} lies in the ideal $\IRW \subset \QQ[\kappa_i \mid i \in \NN]$ generated by such elements. 
\end{enumerate}

 \newcommand{\polyalgn}[1]{
     \QQ\!\left[ \intclass{ij}, \gpsi{e}{i}, \kappa_l            \left| {\substack{1\leq i < j \leq #1 \\ 1\leq l < \infty}} \right.
     \right]
 }
 \medskip 
 To complete the proof, it remains to show that the ideal $\IRW$ does not depend on the value of $d$.  Any tautological class in $H^*(\Mgn{n}; \QQ)$ that appears in the above construction (and any tautological class that makes sense for $d=1$) is in the image of the polynomial algebra $\polyalgn{n}$. The pushforward maps factor through these polynomial algebras. That is to say, there is a map $\mu$ that makes the following diagram commute.
 \[ \xymatrix{
      \polyalgn{n} \ar[rr] \ar[d]^{\mu}
   && H^*(\Mgn{n}; \QQ) \ar[d]^{\left(\pi^{\{1, \ldots, n\}}_{\{1, \ldots, n-1\}}\right)_!}
   \\ \polyalgn{n-1} \ar[rr]
   && H^*(\Mgn{n-1}; \QQ)
 }\]
 
 This map $\mu$ is determined by Procedure~\ref{proced:rules-for-integration}, and does not depend on the value of $d$ (in fact, only the value of $\kappa_e=\chi=2-2g$ is at all affected by what the fiber of our bundle is).  The expressions for further pushforwards such as $\left(\pi^{\{1, \ldots, n\}}_\emptyset\right)_!(b) \in H^*(\Mg)$ also cannot depend on $d$, since they can be computed by applying Procedure~\ref{proced:rules-for-integration} repeatedly.  It follows that the expressions for the generators of the ideal $\IRW$ do not depend on $d$, and thus all of Randal-Williams' examples hold verbatim in the $2d$-dimensional case whenever $d \geq 1$ is odd. 
\end{proof}

\appendix
\secone{MMM classes related to low Pontryagin classes}\label{sec:low-pontryagins}
In this appendix, we discuss of the images of the maps $\taut_d$, $\taut_d'$, and $\taut_{\delta,d}$ defined in Section~\ref{sec:intro-fixed-disks}. We prove that the image of $\taut_{\delta,d}$ is finitely generated. From now on, we omit the subscript $d$ from the notation.

\begin{proposition}\label{pro:pullback-kappa-fixed-disk}
    The maps $\taut$, $\taut'$, and $f^*$ pictured in diagram~\eqref{eq:taut-taut-prime-diagram} are related as follows:
    \begin{enumerate}
                 \item There are classes $q_1 , \ldots, q_{\ceil{\frac{d+1}{4}}-1} \in \image(\taut) \subset H^*(\BDiff M_g; \QQ)$ that generate $\image(\taut)$ as an $\image(\taut')$-module. 
        \item For all $i$, $f^*(q_i) = 0 \in H^*\left(\BDiff(M_g, D^{2d}); \QQ\right)$. 
    \end{enumerate}
            \end{proposition}

\begin{proof}
Let $\pi\co U \to \BDiff M_g$ be the universal bundle and $p_i \in H^*(U; \QQ)$ be the Pontryagin classes of the vertical tangent bundle. Since $M_g$ is $(d-1)$-connected, the map $\pi^*\co H^*\left( \BDiff M_g; \QQ \right) \to H^*(U; \QQ)$ is an isomorphism in degrees $* < d$ (this can be seen e.g. using the Serre spectral sequence). It follows that there are classes $q_i \in H^*(\BDiff M_g; \QQ)$ such that $p_i = \pi^*(q_i)$ for all $i < \ceil{\frac{d+1}{4}}$. 

Now, let $m \in \tbasisl$. If $\deg m \leq 2d$, $\gkappa{m} = 0$ or $\gkappa{m} \in \QQ$, so $\gkappa{m} \in \image{\taut'} \subset \image{\taut}$. If $\deg m > 2d$, then $m$ can be decomposed as a product of some $n \in \tbasiss$ and some Pontryagin classes $p_i$ with $i < \ceil{\frac{d+1}{4}}$. Since the pushforward is a map of $H^*(\BDiff M_g; \QQ)$-modules, $\gkappa{m} = \pi_!(n \cdot \prod \pi^*(q_i)) = \gkappa{n} \cdot \prod q_i$ for some $i$'s. In other words, the $q_i$'s generate $\image(\taut)$ as an $\image(\taut')$-module, as desired.

Let us now prove that that $f^*(q_i)=0$ for all $i$'s. It is sufficient to consider the universal bundle with a fixed disk and prove that the corresponding universal classes $q_i \in H^*\left(\BDiff(M_g, D^{2d}); \QQ\right)$ are zero. We can fix a basepoint $b \in D^{2d} \subset M_g^{2d}$ that determines a section of the universal bundle (which we denote $U_\delta$). The following diagram describes  the corresponding map on cohomology.
\[\xymatrix{
U_\delta = \EDiff(M_g, D^{2d}) \times_{\Diff(M_g, D^{2d})} M_g \ar@/^/[d]^{\pi} 
& H^*(U_\delta; \QQ) 
\ar@/_/[d]_{s^*}
\\
\BDiff(M_g, D^{2d}) \ar@/^/[u]^{s} 
& H^*\left(\BDiff(M_g,D^{2d}); \QQ\right) \ar@/_/[u]_{\pi^*}
}\]

As $s$ is a section we must have $s^*(p_i) = s^*(\pi^*(q_i)) = q_i$ as long as $i < \ceil{\frac{d+1}{4}}$. So, $q_i = s^*(p_i)$ is a characteristic class of the bundle $s^*\left(T_{\pi} U_\delta\right) $ over $\BDiff(M_g, D^{2d})$. Since a neighborhood of the point $b$ is fixed by the action of $\Diff(M_g, D^{2d})$, this bundle is trivial, and so $q_i$ must be zero. 
\end{proof}

\begin{observation}\label{obs:low-pontr-relation}
    For $d>3$, in the notation of the proof above, $p_1 = \pi^*(q_1) \in H^*(U)$. Therefore, for all $g$,
    \[\chi \gkappa{e^2p_1} = \chi \pi_!\left( e^2 \cdot \pi^*(q_1) \right) = \pi_!(e) \cdot q_1 \cdot \pi_!(e^2) =\gkappa{ep_1}\gkappa{e^2}   \in H^*(\BDiff M_g; \QQ).\]
So, the map $\taut$ has non-trivial relations in its kernel that do not depend on $g$. This cannot happen in $\ker \taut_\delta$ or $\ker \taut'$ by Fact~\ref{fact:stability}.
\end{observation}

Proposition~\ref{pro:pullback-kappa-fixed-disk} implies the following.
\begin{corollary}
    If $\gkappa{m} \in \image(\taut) - \image(\taut')$, then $f^*(\gkappa{m}) = 0$. So, $\image(f^* \circ \taut) = \image(\taut_\delta)$.
\end{corollary}
 \begin{theorem}\label{thm:finite-generation-bdry}
    The image of $\taut_{\delta, d}$ is a finitely-generated as a $\QQ$-algebra when $d$ is odd and $g > 1$.
\end{theorem}
\begin{proof}
    By the above Corollary, the image of the map $\taut_\delta$ is a quotient of the image of the map $\taut$, which is finitely generated by Theorem~\ref{thm:finite-generation}.
\end{proof}

\begin{remark}
    If we require that all the Pontryagin classes $p_i$ mentioned in Section~\ref{sec:using-big-theorem} satisfy $i \geq \ceil{\frac{d+1}{4}}$, all of the arguments in that section will apply to the map $\taut'\co \QQ[ \gkappa{p} \mid p \in \tbasiss] \to H^*(\BDiff(M_g); \QQ)$ without any further modification. This way, one can prove that the image of the map $\taut'$ is also finitely generated. That gives another proof that the image of $\taut_{\delta}$ is finitely generated.
\end{remark}

\secone{The Pontryagin-Thom pushforward}\label{sec:pushforward-madness}

While the definition of the pushforward map used throughout this paper applies to all oriented Serre fibrations, in the case of manifold bundles ($M$ is a smooth closed oriented manifold and $\pi\co E \to B$ is a bundle with structure group $\Diff M$), there is another commonly used definition of the pushforward map $\pi_{!PT}\co H^{*+m}(E; \ZZ) \to H^*(B; \ZZ)$ that uses the Pontryagin-Thom construction, see~\cite{boardman69} or~\cite[\S 4]{becker-gottlieb75}.  This \defterm{Pontryagin-Thom pushforward} has the advantage of being defined even for generalized cohomology theories if the bundle has an appropriate orientation. It is also necessary for constructing the kappa classes as pullbacks of natural classes in the cohomology of the infinite-loop space $\Omega^\infty MTSO(2d)$ in the manner of~\cite{madsetillm-stable}.  While we do not use that construction explicitly, it is needed in the proof of Fact~\ref{fact:stability}.

It is conceivable that the notion of kappa classes depends on which definition of the pushforwards one uses. We do not know whether $\pi_!$ and $\pi_{!PT}$ coincide for integral cohomology when $B = \BDiff M$. However, the following fact applies in most relevant cases. It is accepted in the literature, but we provide a proof for completeness.

\begin{proposition}
If $E \to B$ is a manifold bundle with structure group $\Diff M$ and $B$ is a CW complex of finite type, the pushforwards $\pi_{!PT}$ and $\pi_!$ coincide. 

In rational cohomology, $\pi_{!PT}$ and $\pi_!$ coincide for any CW complex $B$.
\end{proposition}
\begin{proof}
One can check that the Pontryagin-Thom construction commutes with bundle pullbacks in an appropriate way so that $\pi_{!PT}$ satisfies the naturality property (2) from Proposition~\ref{pro:properties-pushforward}. If we either work in rational cohomology or assume that $B$ is a CW complex of finite type, we have (see e.g.~\cite[\S 3.F]{hatcher-at} for an overview)
\[H^*(B) = \lim_{\substack{\longleftarrow \\ B' \subset B \\ \text{ finite subcomplex}}} H^*(B').\]
So, we can assume without loss of generality that $B$ is a finite CW complex. Finally, we use the Lemma~\ref{lem:CW-complex-retract-of-manifold} below to reduce the case of a finite CW complex to the case of $B$ a closed oriented manifold. 

    In the case when $B$ is a closed oriented manifold, the fact that $\pi_{!PT}$ and $\pi_!$ coincide is proven in~\cite{boardman69}. Briefly, Boardman proves a multiplicativity property for the \emph{cap} product, similar to property (1) from Proposition~\ref{pro:properties-pushforward}, for both $\pi_!$ and $\pi_{!PT}$. He then deduces that both pushforwards must coincide with the pushforward determined by Poincar\'e duality. 
\end{proof}

\begin{lemma}\label{lem:CW-complex-retract-of-manifold}
Any finite CW complex $B$ is a retract of a smooth oriented closed manifold $D$. In particular, there is a map $f\co D \to B$ such that $f^*\co H^*(B; \ZZ) \to H^*(D; \ZZ)$ is injective.
\end{lemma}
\begin{proof}\hspace{-0.7ex}\footnote{We thank Alexander Kupers for a key idea for this proof. This argument is also in~\cite{thom54}.}  It is possible to embed $B$ into a Euclidean space. A sufficiently small tubular neighborhood $T$ of such an embedding will be an oriented compact manifold \emph{with boundary} that deformation retracts onto $T$ (see e.g. appendix to~\cite{hatcher-at}). In particular, we have maps $B \stackrel{i}{\hookrightarrow} T \stackrel{f'}{\to} B$ such that the composition is the identity.

Let $D = T \sqcup_{\delta T} (-T)$ be the double of $T$. It is a closed oriented manifold. There is an obvious inclusion $T \hookrightarrow D$ and, crucially, the map $f'\co T \to B$ extends to a map $f\co D \to B$. So, we have our retraction
\[\xymatrix{
B \ar@{^(->}[r]^{i}  &T \ar@{^(->}[r] \ar@/_1.5em/[rr]^{f'} & D \ar[r]^{f} & B.
}\]
The composition is the identity since it coincides with $f' \circ i$.
\end{proof}

\mybibstyle
\bibliography{bib}
 
\end{document}